\documentclass[12pt]{amsart}
\usepackage{amsmath,amsthm,amsfonts,amssymb,latexsym,enumerate}
\usepackage[pagebackref]{hyperref}
\usepackage{color, xcolor}
\usepackage{enumerate, enumitem, verbatim, graphicx, longtable}

\newcommand{\FF}{{\mathbb{F}}}

\newcommand{\QQ}{{\mathbb{Q}}}

\newcommand{\bC}{{\mathbf C}}

\newcommand{\bZ}{{\mathbf Z}}
\newcommand{\bO}{{\mathbf O}}
\newcommand{\bF}{{\mathbf F}}
\newcommand{\bN}{{\mathbf N}}
\newcommand{\bE}{{\mathbf E}}

\newcommand{\bH}{{\mathbf H}}

\newcommand{\cJ} {\mathcal J}

\newcommand{\SSS}{\mathsf{S}}
\newcommand{\AAA}{\mathsf{A}}
\newcommand{\CCC}{\mathsf{C}}
\newcommand{\DDD}{\mathsf{D}}

\newcommand{\bbS}{\mathbb{S}}

\newcommand{\Aut}{{{\operatorname{Aut}}}}

\newcommand{\Irr}{{{\operatorname{Irr}}}}
\newcommand{\Out}{{{\operatorname{Out}}}}
\newcommand{\GL}{\operatorname{GL}}

\newcommand{\PSL}{\operatorname{PSL}}
\newcommand{\PSp}{\operatorname{PSp}}
\newcommand{\Sp}{\operatorname{Sp}}

\newcommand{\SL}{\operatorname{SL}}
\newcommand{\SO}{\operatorname{SO}}

\newcommand{\Syl}{\operatorname{Syl}}
\newcommand{\Ker}{\operatorname{Ker}}
\newcommand{\Gal}{\operatorname{Gal}}

\newcommand{\tw}[1]{{}^{#1}\!}

\newcommand{\sym}{\SSS}

\newtheorem{thm}{Theorem}[section]
\newtheorem{lem}[thm]{Lemma}
\newtheorem{con}[thm]{Conjecture}
\newtheorem{pro}[thm]{Proposition}
\newtheorem{cor}[thm]{Corollary}

\newtheorem*{conA'}{Conjecture A'}

\newtheorem{thml}{Theorem}

\newtheorem{corl}[thml]{Corollary}

\theoremstyle{definition}

\numberwithin{equation}{section}

\def\irr#1{{\rm Irr}(#1)}
\def\aut#1{{\rm Aut}(#1)}
\def\cent#1#2{{\bf C}_{#1}(#2)}

\def\norm#1#2{{\bf N}_{#1}(#2)}

\def\zent#1{{\bf Z}(#1)}

   \def \mod#1{\, {\rm mod} \, #1 \, }

\newcommand{\type}{\operatorname}
\newcommand{\wt}{\widetilde}

\begin{document}

\title[A normal version of Brauer's height zero conjecture]{A normal version of Brauer's height zero conjecture}

\author{Alexander Moret\'o}
\address[A. Moret\'o]{Departament de Matem\`atiques, Universitat de Val\`encia, 46100
  Burjassot, Val\`encia, Spain}
\email{alexander.moreto@uv.es}

\author{A. A. Schaeffer Fry}
\address[A. A. Schaeffer Fry]{Dept. Mathematics - University of Denver, Denver, CO 80210, USA}
\email{mandi.schaefferfry@du.edu}


\thanks{We thank Gunter Malle for his comments on an earlier draft of this manuscript, as well as for helpful conversations on this problem during the  Isaac Newton Institute
program ``Groups, Representations, and
Applications: New Perspectives", supported by EPSRC grant EP/R014604/1, where this project was started.  The first author is supported by Ministerio de Ciencia e Innovaci\'on (Grants PID2019-103854GB-I00  and PID2022-137612NB-I00 funded by MCIN/AEI/10.13039/501100011033 and ``ERDF A way of making Europe"). He also acknowledges support by Generalitat Valenciana CIAICO/2021/163.  The second author  gratefully acknowledges support from the National Science Foundation, Award No. DMS-2100912,  and her former institution, Metropolitan State University of Denver, which holds the award and allows her to serve as PI}



\date{\today}

\begin{abstract}
The celebrated It\^o--Michler theorem asserts that a prime $p$ does not divide the degree of any irreducible character of a finite group $G$ if and only if $G$ has a normal and abelian Sylow $p$-subgroup. The principal block case of the recently-proven Brauer's height zero conjecture  isolates the abelian part in the It\^o--Michler theorem. In this paper, we show that the normal part can also be isolated in a similar way. This is a consequence of work on a strong form of the so-called Brauer's height zero conjecture for two primes of Malle and Navarro. Using our techniques, we also provide an alternate proof of this conjecture.
\end{abstract}

\maketitle



\section{Introduction}

The It\^o--Michler theorem \cite{mich86} was one of the first results obtained as a consequence of the announcement of the classification of finite simple groups. This fundamental result in representation theory of finite groups asserts that a prime $p$ does not divide the degree of any complex irreducible character of a finite group $G$ if and only if $G$ has a normal abelian Sylow $p$-subgroup. 
This theorem was motivated by Brauer's height zero conjecture, which was proposed in 1955, long-before the classification was completed, and was recently proved in \cite{mnst}. Brauer's height zero conjecture can be considered as  a block version of the It\^o--Michler theorem, asserting that if $B$ is a Brauer $p$-block of $G$, then every irreducible character in $B$ has height zero  if and only if the defect group of the block is abelian. In particular, the principal block case proven in \cite{mn2} says that a Sylow $p$-subgroup is abelian if and only if $p$ does not divide the degree of any irreducible character in the principal $p$-block $B_p(G)$ of $G$.

Noting that Brauer's height zero conjecture isolates the ``abelian" part of the It\^o--Michler theorem, isolating the ``normal" part is another interesting problem (see, e.g. \cite{nav16}) that has also received attention.  As pointed out by Navarro in \cite[p. 1365]{nav16} ``when we restrict our attention to the characters in the principal block then the commutativity of $P$ is characterized (by the principal block case of Brauer's height zero conjecture) while the normality of $P$ is lost". Our first main result shows  that, perhaps surprisingly, the normality of $P$ is also characterized when we restrict our attention to the characters in the principal block {\it for the primes different from $p$}.

\begin{thml}\label{thm:normalsylow}
Let $G$ be a finite group, let $p$ be a prime and let $P\in\Syl_p(G)$.
Then $P\trianglelefteq G$ if and only if for every prime $q\neq p$ that divides $|G|$  and every $\chi\in\Irr(B_q(G))$, $p$ does not divide $\chi(1)$.
\end{thml}

This is therefore a dual version of Brauer's height zero conjecture for principal blocks.   We remark that Malle and Navarro found a different subset of $\Irr(G)$ that characterizes the normality of a Sylow $p$-subgroup. In \cite{mn12} they proved that $P\trianglelefteq G$ if and only if $p$ does not divide $\chi(1)$ for every irreducible constituent $\chi$ of the permutation character $(1_P)^G$. (See also \cite{GLLV22} for a further refinement of the statement.) Note, however, that while the union of $\Irr(B_q(G))$ for $q\neq p$ prime divisor of $|G|$  can be easily determined from the character table of $G$, it is still not known whether $(1_P)^G$ is determined by the character table of $G$.
(However, see \cite{n98} or  \cite[Thm.~1.2]{nr}, where this is proved for $p$-solvable groups.)

Using Brauer's height zero conjecture for principal blocks \cite{mn2} and Theorem  \ref{thm:normalsylow}, we deduce the following refinement of the It\^o--Michler theorem.

\begin{corl} \label{cor:blockitomichalph}
Let $G$ be a finite group and let $p$ be a prime. Then $G$ has a normal and abelian Sylow $p$-subgroup if and only if $p$ does not divide $\chi(1)$ for every $\chi$ that belongs to some principal block for some prime divisor of $|G|$.
\end{corl}

This project was motivated by a conjecture of Malle and Navarro  called Brauer's height zero conjecture for two primes: if $G$ is a finite group and $p$ and $q$ are different prime divisors of $|G|$,  then $G$ has a nilpotent Hall $\{p,q\}$-subgroup if and only if  $p$ does not divide the degree of the characters in the principal $q$-block and $q$ does not divide the degrees of the characters in the principal $p$-block. (Actually, they allowed the case $p=q$ in their statement, which is exactly Brauer's height zero conjecture for principal blocks, proved in \cite{mn2}.) The ``only if" direction of this conjecture was proved in \cite{mn}, while the ``if" direction has recently been  proven in \cite{lwwz}.
This is one of the recent results in modular representation theory that take two primes into account (see, for instance, \cite{nrs, mn, lwwz}).

In \cite{mn}, Malle and Navarro pointed out that it is an interesting question to characterize when the irreducible characters of the principal $p$-block $B_p(G)$ have degree not divisible by $q$. Our original goal was to address this question. We obtain a block version of the It\^o--Michler theorem.

\begin{thml}\label{thm:pnilphall}
Let $G$ be a finite group and let $p$ and $q$ be two different primes. 
Assume that $S$ is not a composition factor of $G$ if $(S,p,q)$ is one of the $5$ exceptions listed in Theorem \ref{sporadic} below.
If $q$ does not divide the degree of any irreducible character in $B_p(G)$, then  a Sylow $p$-subgroup of $G$ normalizes a Sylow $q$-subgroup of $G$.
\end{thml}

We remark that the conclusion of Theorem \ref{thm:pnilphall} is equivalent to $G$ having a $p$-nilpotent Hall $\{p,q\}$-subgroup. There does not seem to be much room for improvement in Theorem \ref{thm:pnilphall}. We will see in Theorem \ref{sporadic} below that the exceptions listed are necessary.  It is not clear how to strengthen the conclusion either. As $G=\PSL_2(11)$ for $p=2$ and $q=3$ shows, the hypothesis that $q$ does not divide the degrees of the characters in the principal $p$-block does not imply that $G$ has a nilpotent Hall $\{p,q\}$-subgroup. In fact, in this example $G$ has a $p$-nilpotent Hall $\{p,q\}$-subgroup isomorphic to $\DDD_{12}$ but it also has Hall $\{p,q\}$-subgroups isomorphic to $\AAA_4$ which are not $p$-nilpotent. This shows that the hypothesis does not imply that Hall subgroups are conjugate.

Further, the converse of Theorem \ref{thm:pnilphall} does not hold.
There are many simple groups with a Sylow $p$-subgroup that normalizes a Sylow $q$-subgroup  but such that $q$ divides the degree of some irreducible character in the principal $p$-block, for instance $\AAA_5$ for $p=3$ and $q=2$.
The $\{p,q\}$-separable case of Theorem \ref{thm:pnilphall} for arbitrary blocks is the main result of \cite{nw}. As pointed out in \cite{mn}, the group $6.\AAA_7$ shows that the conclusion of Theorem \ref{thm:pnilphall} does not hold for arbitrary non-principal $p$-blocks of full defect.

Theorem \ref{thm:normalsylow}, and hence Corollary \ref{cor:blockitomichalph}, are pleasant consequences of Theorem \ref{thm:pnilphall}. Using the techniques developed to prove Theorem \ref{thm:pnilphall}, we further provide a short proof of Brauer's height zero conjecture for two primes (see Theorem \ref{thm:2pBHZ} below). Therefore, our work does not rely on \cite{lwwz}.

In Section 2, we prove Theorem \ref{thm:pnilphall} for almost simple groups.  In Section 3, we  obtain some other  results that will be used in the proof of Theorem \ref{thm:pnilphall}. We complete our proof of Brauer's height zero conjecture for two primes in Section 4  and the proof of Theorems \ref{thm:normalsylow} and \ref{thm:pnilphall} in Section 5. We also obtain the principal block case of a theorem that Navarro \cite{nav04} proved in 2004 assuming the Alperin--McKay conjecture; see Theorem \ref{main}. We conclude in Section 6 with a version of Brauer's height zero conjecture for any number of primes and conjecturing a Galois version and a height-zero version of Theorem \ref{thm:normalsylow}, which we prove for solvable groups.

\section{Almost simple groups}

Throughout, we continue to let $B_p(G)$ denote the principal $p$-block of a group $G$. 
If $N$ is a normal subgroup of a group $G$ and $\theta$ is an irreducible character of $N$, we write $\Irr(G|\theta)$ to denote the set of irreducible constituents of the induced character $\theta^G$.
Our notation is standard and follows \cite{isa, isa2, nav}. 
Here we begin our discussion on almost simple groups.

First, we consider sporadic and alternating groups.

\begin{thm} \label{sporadic}

Let $G$ be an almost simple group with socle a sporadic simple group $S$ or the Tits group $\tw{2}\type{F}_4(2)'$ and let $p$ and $q$ be two different primes dividing $|G|$. Then the principal $p$-block of $G$ has an irreducible character of degree divisible by $q$ unless $(S, p, q)$ is one of the following: $(S,p,q)=(J_1,2,3)$, $(J_1,2,5)$, $ (J_1,3,5)$, $ (J_1,5,3)$, $ (J_4,5,7)$, $ (J_4,7,5)$, $ (J_4, 7,11)$, $ (M_{11},5,3),$ or $(M_{22},7,2)$. 
Among these exceptional cases, $G$ does not have  a $p$-nilpotent Hall $\{p,q\}$-subgroup if and only if  $(S,p,q)$ is in the following list:
\begin{itemize}
\item $(J_1,2,3)$,
\item $(J_1,2,5)$,
\item $(J_4,7,11)$,
\item $(M_{11},5,3)$, or
\item $(M_{22}, 7,2)$.
\end{itemize}
\end{thm}

\begin{proof}
The first part can be checked with GAP \cite{gap}.
The second part can be checked with GAP, \cite{wil}, and \cite{vr}. Note that by  \cite[Thm.~8.2]{vr}, $M_{22}$ does not contain a Hall $\{7,2\}$-subgroup, and hence neither does $M_{22}.2$.
\end{proof}

\begin{thm}
\label{alt}
Let $G$ be an almost simple group with socle an alternating group and let $p$ and $q$ be two different primes dividing $|G|$. Then the principal $p$-block contains an irreducible character of degree divisible by $q$.
\end{thm}

\begin{proof}
Let $S=\AAA_n$. If $n=6$ this can be checked with GAP. If $n\neq 6$ this was proved in \cite[Thm. C]{gmv} when $q<p$ and in  \cite[Thm. A]{gm} when $q>p$.
\end{proof}

We now turn our attention to groups $S$ of Lie type. First, we consider the case that $r_0\in\{p,q\}$, where $r_0$ is the defining characteristic for $S$.

\begin{lem}
\label{lem:liedefining}
Let $p$ and $q$ be two different primes.
Let $S$ be a simple group of Lie type such that $p$ and $q$ divide $|S|$ and $S$ is defined in characteristic $p$ or $q$. Then there exists $\chi\in\irr{B_p(S)}$ of degree divisible by $q$.
\end{lem}
\begin{proof}
First, assume that $S$ is defined in characteristic $p$.  Then $B_p(S)=\Irr(S)\setminus\{\mathrm{St}_S\}$, where $\mathrm{St}_S$ is the Steinberg character, by \cite[Thm. 6.18]{CE04}.  Since $S$ is simple, the It\^o--Michler theorem \cite[Thm. 2.3]{mich86} tells us that either $q\nmid |S|$, or $q\mid\chi(1)$ for some $\chi\in\irr{S}$ (and hence some $\chi\in\irr{B_p(S)}$, since $\mathrm{St}_S$ has degree a power of $p$).

Next, suppose that $S$ is defined in characteristic $q$. Then by \cite[Thm. 6.8]{malle07}, either every nontrivial unipotent character of $S$ has degree divisible by $q$, or $q\leq 3$  and $S$ is $\type{B}_n(2)\cong \type{C}_n(2)$ with $n\geq 3$, $\type{B}_2(2)'$, $\type{G}_2(2)'$, $\type{F}_4(2)$, or $\type{G}_2(3)$. Since $B_p(S)$ necessarily contains a nontrivial unipotent character (see, e.g. \cite[Lem.~3.15]{mnst}), we may assume that $S$ is one of these exceptions. If $S=\type{B}_2(2)'$, $\type{B}_3(2)$, $\type{G}_2(2)'$, $\type{F}_4(2)$, or $\type{G}_2(3)$, we may check in GAP that $B_p(S)$ contains a character of degree divisible by $2$, respectively $3$, for every other prime $p$ dividing $|S|$. So, assume $S=\type{B}_n(2)$ with $n\geq 4$. Here \cite[Thm. 6.8]{malle07} yields that there are only five unipotent characters of odd degree, but $B_p(S)$ contains $k(2e, \lfloor \frac{n}{e}\rfloor)\geq 6$ unipotent characters (see, e.g., the discussion before \cite[Prop. 5.4]{malle17}), where $e$ is the order of $4$ modulo $p$ and $k(2e, w)$ is the number of irreducible characters of the imprimitive complex reflection group $G(2e, 1, w)$.  Hence we see $B_p(S)$ must contain a unipotent character of degree divisible by $2$, as desired.
\end{proof}

For the remainder of this section, the primary case is then  that neither $p$ nor $q$ are the defining characteristic for $S$. We will rely on and build upon the work of \cite{mn} in this situation.

For an integer $m$ relatively prime to  a prime $p$, we will let $d_p(m)$ denote the order of $m$ modulo $p$ if $p$ is odd and modulo $4$ if $p=2$.

\subsection{Exceptional Groups}
We begin with the case of exceptional types, but excluding Suzuki and Ree groups for now. We remark that, throughout, $\type{E}^\epsilon_6(r)$
 will denote $\type{E}_6(r)$ when $\epsilon=1$ and $\tw{2}\type{E}_6(r)$ when $\epsilon=-1$.
 
\begin{lem}\label{lem:exceptinitial}
Let $H=\bH^F$ be a quasisimple group of Lie type such that $\bH$ is a simple, simply connected algebraic group of exceptional type (including the case $H=\tw{3}\type{D}_4(r)$) and $F\colon \bH\rightarrow\bH$ is a Frobenius endomorphism defining $H$ over $\FF_r$. Let $G=H/Z$ for some $Z\leq \zent{H}$. Let $p$ and $q$ be two primes dividing $|G|$. Assume that there is no  $\chi\in\irr{B_p(G)}$ with $q\mid\chi(1)$. Then neither $p$ nor $q$ divides $r$ and: 

\begin{enumerate}
\item $q$ is odd;
\item $d_p(r)=d_q(r)$;
\item  a Sylow $q$-subgroup $Q$ of $G$ is abelian; and
\item there is a Sylow $p$-subgroup $P$ of $G$ such that $P$ normalizes $Q$.
\item  Further, either there is a Sylow $p$-subgroup $P$ of $G$ such that $[P,Q]=1$ or $(H, d_q(r)=d_p(r), p)$ is one of:
\begin{itemize}
\item $(\type{E}_6(r), 2, 2)$
\item $(\tw{2}\type{E}_6(r), 1, 2)$
\item $(\tw{3}\type{D}_4(r), 1, 3)$
\item $(\tw{3}\type{D}_4(r), 2, 3)$

\end{itemize}
\end{enumerate}
\end{lem}
\begin{proof}
Throughout, the notation for unipotent characters is taken from \cite[Sec.~13.9]{carter}. Further, we use $\Phi_i:=\Phi_i(r)$ to denote the $i$th cyclotomic polynomial in $r$.

Write $S:=H/\zent{H}$, and note that $G\in\{H, S\}$ in these cases. Further, letting $H^\ast=(\bH^\ast)^F$ with $(\bH^\ast, F)$ dual to $(\bH, F)$ (and hence $\bH^\ast$ is simple of adjoint type), we have $S\cong [H^\ast, H^\ast]$ and $\wt{S}\cong H^\ast$, where $\wt{S}$ is the group of inner-diagonal automorphisms of $S$. (See, e.g. \cite[Prop.~24.21, Tab.~22.1]{MT11} and the subsequent discussions, since $S$ is of exceptional type.) Since we have $\Irr(B_p(S))\subseteq \Irr(B_p(G))$, viewed by inflation, we have by Lemma \ref{lem:liedefining} that $r=r_0^f$ for some prime $r_0\not\in\{p,q\}$. Let $d_p:=d_p(r)$ and $d_q:=d_q(r)$.

Now, we claim that  $q$ is odd. Indeed, if $q=2$, then since the group $[H^\ast, H^\ast]$ is simple,  \cite[Thm.~2.2]{BFMMNSST} yields that there is a $p$-element $s\in [H^\ast, H^\ast]$ such that its centralizer in $[H^\ast, H^\ast]$ has even index. 
Since this group is index $2$ or odd in $H^\ast$, the same is true for the centralizer in $H^\ast$. Let $\chi_s$ be the semisimple character of $H$ corresponding to the class of $s$, so that $\chi_s(1)=[H^\ast:\cent{H^\ast}{s}]_{r'}$ is even. 
Further, $\chi_s$  is trivial on $\zent{H}$ by \cite[Lem.~4.4(ii) and Rem.~4.6]{NT13} and lies in $B_p(H)$ using \cite[Cor.~3.4]{hiss90} if $\zent{\bH}$ is connected or \cite[Thm.~21.13]{CE04}  if $p$ is good for $\bH$ and $\cent{\bH^\ast}{s}$ is connected. Namely, the latter holds when $p\nmid|\zent{H}|$ by \cite[Exer.~20.16]{MT11}. So, we are left to consider the case $p=3$ and $S=\type{E}_6(r)$ or $\tw{2}\type{E}_6(r)$.
In these cases, we see by \cite[Sec.~13.9]{carter} that there is a unipotent character whose degree is divisible by $2$ but not by $3$. (Namely, we may take the character $\phi_{20,2}$ and $\phi_{4,1}$, 
 respectively.) But since $B_3(H)$ is the unique unipotent block containing unipotent characters of $3'$-degree  (this follows from \cite[Thm.~6.6]{malle07} and \cite[Thm.~A]{Eng00}) and unipotent characters are trivial on the center, we see these lie in $B_3(G)$.

So, we see that $q$ is odd. Now,  consider the case that $q\mid |\zent{H}|=|\wt{S}/S|$.
Then $S=\type{E}_6^\epsilon(r)$ with $q=3\mid (r-\epsilon)$. If $S=\type{E}_6(r)$ with $q=3\mid (r-1)$, we see that the characters $\phi_{81, 10}, \phi_{6.1}$, and $\phi_{90, 8}$ in the notation of \cite[Sec.~13.9]{carter} have degree divisible by $3$. However, using the explicit decompositions in \cite[Tables 1-2]{BMM} when $d_p$ is non-regular, and the fact that $B_p(H)$ is the unique unipotent block containing unipotent characters of $p'$-degree if $d_p$ is regular (again this follows from \cite[Thm.~6.6]{malle07} and \cite[Thm.~A]{Eng00}), we see that for each $p$, at least one of these three characters lies in $B_p(H)$. When $S=\tw{2}\type{E}_6(r)$ with $3\mid (r+1)$, the same holds, with the characters $\phi_{9,6}'', \phi_{2,4}', \phi_{6,6}''$.

Now, assume that $q\nmid |\zent{H}|=[\wt{S}: S]$, so that a Sylow $q$-subgroup $Q$ of $S$ can be identified with one of $H$ and of $H^\ast\cong \wt{S}$.  Suppose that  there exists some $p$-element $s\in H$ such that the corresponding semisimple character 
$\chi_s\in\irr{\wt{S}}$ satisfies $q\mid \chi_s(1)$. By \cite[Cor.~3.4]{hiss90}, we have $\chi_s\in\irr{B_p(\wt{S})}$.
Then since $q\nmid |\wt{S}/S|$, we obtain that there is some $\chi\in\irr{B_p(S)}\subseteq \irr{B_p(G)}$ with $q\mid \chi(1)$. 
Since $\chi_s(1)=[H: \cent{H}{s}]_{r'}$, we may assume that every $p$-element $s\in H$ centralizes a Sylow $q$-subgroup of $H$. It follows by \cite[Thm.~5.9]{malle07} that every such $s$ also centralizes a Sylow $d_q$-torus of $(\bH, F)$.

Now, by \cite[Prop.~2.12]{BFMMNSST} and also applying \cite[Prop.~2.2]{malle14}, we have either $Q$ is abelian, 
  $\cent{H}{Q}=\zent{H}\zent{Q}$, or $(H,q)$ is as in the list \cite[Prop.~2.12]{BFMMNSST}(b)-(g) and $d_p, d_q\in\{1,2\}$, with $d_p=d_q$ unless possibly if $p=2$. If $\cent{H}{Q}=\zent{H}\zent{Q}$, we obtain every $p$-element $s\in H$ lies in $\zent{H}$, which is not true.

Next, suppose that $(H,q)$ is as in the list \cite[Prop.~2.12]{BFMMNSST}(b)-(g). In each case, we will exhibit a unipotent character whose degree is divisible by $q$ and lies in $B_p(H)$, which will follow from arguments like above and \cite[Thm.~3.10]{BMM} and \cite[Thm.~A]{KM15}.
In case (b),(c), we have $q=5\mid (r-\epsilon)$ and $H=\type{E}_6^\epsilon(r)$ and the characters $\phi_{20,2}$ or $\phi_{4,1}$, in the respective cases $\epsilon=1, \epsilon=-1$, lie in $B_p(H)$ and have degree divisible by $5$.  In case (d), we have $(H,q)=(\type{E}_7(r), 3)$, and we see $\phi_{27,2}$ has degree divisible by $3$ and must lie in $B_p(H)$.
In case (e), we have $(H,q)=(\type{E}_7(r), 5)$, and $\phi_{105, 15}$ has degree divisible by $5$ and lies in $B_p(H)$.
In case (f) and (g), we have $(H,q)=(\type{E}_7(r), 7)$, resp. ($\type{E}_8(r), 7)$, and we see $\phi_{7,1}$, resp. $\phi_{35, 2}$,  has degree divisible by $7$ and  lies in $B_p(H)$.

It follows that $Q$ is abelian. 
Suppose first that $P$ is also abelian. Note that both $p$ and $q$ divide a unique cyclotomic factor of the generic order of $H$, by \cite[Lem.~2.1, Prop.~2.2]{malle14}. If $d_p=d_q$, there is a Sylow $d_q$-torus containing a Sylow $p$-subgroup $P$ and a Sylow $q$-subgroup $Q$, so that $[P,Q]=1$.  

%

So, assume $d_p\neq d_q$.
 Then by applying \cite[Prop.~2.2]{malle14}, we have by \cite[Prop.~2.13]{BFMMNSST} that the Sylow $d_q$-tori of $\bH$ are not maximal tori, and hence $d_q$ is as in \cite[Table 1]{BFMMNSST}. Recall our assumption that any $p$-element centralizes a Sylow $d_q$-torus.
 Since $d_p\neq d_q$ and $p$ divides exactly one cyclotomic polynomial in the generic order of $H$, by the structure of centralizers of Sylow $d_q$-tori given in \cite[Table 3.3]{GM20}, we then have $(H, d_q, d_p)$ is as in Table \ref{tab:except1}.
  In the last column, we list a unipotent character in $B_p(H)$ with degree divisible by $q$, which can be seen by arguing in the same ways as above.

 \begin{table}
 \begin{tabular}{|c|c|c|c|}
 \hline
 $H$ & $d_q$ & $d_p$ & $\chi\in\Irr(B_p(H))$ with $q\mid \chi(1)$\\
 \hline
 $\tw{3}\type{D}_4(r)$ & 1 & 3 & $\tw{3}\type{D}_4[1]$\\
 \hline
 $\tw{3}\type{D}_4(r)$ & 2& 6& $\phi_{2,2}$\\
 \hline
 $\type{E}_6(r)$ & 2 or 4 & 1& $\phi_{64,4}$\\
 \hline

 $\type{E}_6(r)$ & 6 & 3& $\phi_{64,4}$\\
 \hline
 $\tw{2}\type{E}_6(r)$ & 1 or 4 & 2& $\tw{2}\type{A}_{5},1$\\
 \hline

 $\tw{2}\type{E}_6(r)$ & 3 & 6& $\tw{2}\type{A}_{5},1$\\
 \hline

 $\type{E}_7(r)$ & 3 & 1& $\phi_{27,2}$\\
\hline
 $\type{E}_7(r)$ & 4 & 1 or 2& $\phi_{168,6}$\\
\hline
 $\type{E}_7(r)$ & 6 & 2& $\phi_{27,2}$\\
 \hline
 \end{tabular}
 \caption{Possible $(H, d_q, d_p)$ when 
  $d_p\neq d_q$ and $\Phi_{d_p}\mid|\cent{H}{\mathbb{S}_{d_q}}|$ for $\mathbb{S}_{d_q}$ a Sylow $d_q$-torus of $\bH$ that is not maximal}\label{tab:except1}
  \end{table}

So now assume that $P$ is nonabelian. Then note that $p\leq 7$. Further, we have $p\leq 3$ if $H\in\{\type{G}_2(r), \tw{3}\type{D}_4(r), \type{F}_4(r)\}$. If $H=\type{E}_6^\epsilon(r)$, we have either $p\leq 3$ or $p=5$ and $d_p=1$ or $2$, depending on whether $\epsilon=1$ or $-1$.  If $H=\type{E}_7(r)$, then $p\leq 7$ and $d_p\in\{1,2\}$. If $H=\type{E}_8(r)$ then $p\leq 7$ and $d_p\in\{1,2\}$ if $p=7$.

First, suppose that  $H=\tw{3}\type{D}_4(r)$. Then there is a unique unipotent $2$-block by \cite[Thm.~21.14]{CE04}, and we see from the character degrees in \cite[Sec.~13.9]{carter} that for each possible $d_q$, there is a unipotent character of degree divisible by $q$. So we assume $p=3$. 
If $d_q\in\{1,2\}$, we are in the situation that a Sylow $d_q$-torus is not maximal, and we see from Table \ref{tab:except1}  that $d_3=d_q$. (This gives one of the situations in part (v) of our conclusion.)
If $d_q=3$ or $6$, we have $d_q$ is a  regular number for $H$ and a Sylow $d_q$-torus $\mathbb{S}_{d_q}$ of $\bH$ is maximal, so that $3=p\mid|\cent{H}{\mathbb{S}_{d_q}}|=|\mathbb{S}_{d_q}^F|$, so $3=p\mid \Phi_{d_q}$. In particular, this means that $d_q=3^ad_3$ for some $a\geq 0$. Then $d_3=1$, respectively $2$. In this case, we see that the element $s_5$, respectively $s_{10}$ in \cite{DM87}, chosen to have $3$-power order larger than $3$, does not centralize a Sylow $q$-subgroup.
If $d_q=12$, we see again $3\mid \Phi_{12}$, a contradiction.


Next, suppose $H=\type{E}_6(r)$. 
 If $d_p=1$, we have the principal series unipotent characters $\phi_{64,4}$, $\phi_{60,5}$, and $\phi_{81, 6}$
must lie in $B_p(H)$, and at least one of these has  degree divisible by $q$ if $d_q\neq 1.$  If $d_p=d_q=1$, we may consider a $p$-element $s$ of maximal order, which does not divide  the exponent of a Sylow $1$-torus, and hence $s$ does not centralize a Sylow $d_q$-torus.

Hence, we may assume that $d_p=2$ and $p\in\{2,3\}$. Suppose that $d_q=2,4,$ or $6$, so that a Sylow $d_q$-torus is not maximal.  
If $d_q=6$, then we see in Table \ref{tab:except1} that $p\mid\Phi_3\Phi_6$ and we have $p=3$. Then applying the Ennola duality proved in \cite[Thm.~3.3]{BMM}, we see it suffices to exhibit a principal series unipotent character for $H^-:=\tw{2}\type{E}_6(r)$ in the case $d_p=1$ whose degree is divisible by $\Phi_3$. The character $\phi_{9,6'}$ of $H^-$ satisfies these properties. If $d_q=4$, then we see $p=2$ and in this case it similarly suffices to note that the character $\phi_{4,1}$ of $H^-$ is a principal series unipotent character with degree divisible by $\Phi_4$.
Now suppose $d_q=2$. When $p=3$, we again see that the maximal order of a $3$-element does not divide the size of the centralizer of a Sylow $2$-torus, whose order polynomial is divisible only by the cyclotomic polynomials $\Phi_1$ and $\Phi_2$. This leaves only the case $p=2$, as indicated in the statement.

Then we may assume that a Sylow $d_q$-torus  is  a maximal torus of $\bH$.  Either $d_q$ is regular, and hence $d_q=p^ad_p$ for some $a\geq 0$ since $p$ divides the order of a self-centralizing Sylow $d_q$-torus, or $d_q=5$.
When $d_q=5$, similar to before   by the Ennola duality \cite[Thm.~3.3]{BMM}, it suffices to note that the unipotent character $\phi_{4,1}$ of $H^-:=\tw{2}\type{E}_6(r)$ is a principal series character in $B_p(H^-)$ when $d_p=1$ whose degree is divisible by $\Phi_{10}$. So we finally suppose that $d_q$ is regular, so $d_q=p^ad_p$.
Since we are also assuming a Sylow $d_q$-torus is maximal, the only option is  $d_q=8$.
Here we again appeal to the case of $H^-$, with $d_q=8$. The principal series unipotent character $\phi_{2,4'}$ of $H^-$ has degree divisible by $\Phi_8$, yielding the result in this case.

The cases for  $H=\tw{2}\type{E}_6(r)$ not yet considered follow from the case of $\type{E}_6(r)$ by the Ennola duality proved in \cite[Thm.~3.3]{BMM} and at times arguing as above with $p$-elements of maximal order.

Next, suppose $H=\type{E}_7(r)$. Recall that again $d_p\in\{1,2\}$ in this case. The case $d_p=1$ is completed very similar to the case of $\type{E}_6(r)$ above. Namely, for each $d_q\neq 1$ such that $\Phi_{d_q}$ divides the order polynomial of $H$, there is a principal series unipotent character of degree divisible by $\Phi_{d_q}$, and when $d_p=d_q=1$, we may again consider a $p$-element of maximal order. The case $d_p=2$ then follows by the Ennola duality proved in \cite[Thm.~3.3]{BMM}. The cases of $H=\type{F}_4(r)$ and $\type{G}_2(r)$ are completed similarly.

Finally, let $H=\type{E}_8(r)$. Here the Sylow $d_q$-torus is maximal. Then if $d_q$ is regular, we have $d_q=p^ad_p$ as before. If  $d_p\in\{1,2\}$, we may argue similarly to the preceeding cases. So, suppose $p=5$ and $d_p=4$. Here, $d_p$ is a regular number, so as before, $B_p(H)$ is the unique block containing $p'$-degree unipotent characters. If $d_q$ is regular, we see that  $d_q\in\{4,20\}$. Again taking $s$ to be a $5$-element of maximal order, we see $|s|$ does not divide the exponent of a Sylow $\Phi_{d_q}$-torus, and hence $s$ does not centralize a Sylow $q$-subgroup.
If $d_q$ is not regular, then $d_q\in\{7, 9, 14, 18\}$, and we see the unipotent character $\phi_{112,3}$ has degree divisible by $\Phi_{d_q}$ if $d_q\neq 9$, but not divisible by $p=5$. Similarly, the unipotent character $\phi_{567,6}$ has degree divisible by $\Phi_9$ but not by $p=5$.

Now, suppose that we are in one of the situations in (v). Let $d_p=d_q=:d$. We may identify $P\in\Syl_p(G)$ and $Q\in\Syl_q(G)$ with Sylow subgroups of $H$. If $p$ is odd, then by \cite[Thm.~4.10.2]{GLS}, we have $ P$ is conjugate to a group of the form $P_T\rtimes P_W$, where $P_T$ lies in a Sylow $d$-torus of $H$ and $P_W$ normalizes that torus. But similarly, $ Q=Q_T$ lies in a Sylow $d$-torus. Then these may be chosen such that $[P_T, Q_T]= 1$ and $P_W$ normalizes $Q_T$. If $p=2$, then we still have $P_T\lhd  P$ and  $[P_T, Q_T]= 1$, and $ P/P_T$ still normalizes the Sylow $d$-torus. This yields (iv), completing the proof.
\end{proof}

We next address the case of Suzuki and Ree groups.

\begin{lem}\label{lem:suzree}
Let $G$ be a simple Suzuki or Ree group $\tw{2}\type{G}_2(3^{2k+1})$, $\tw{2}\type{B}_2(2^{2k+1})$, or $\tw{2}\type{F}_4(2^{2k+1})$. Write $r^2=3^{2k+1}$, respectively $2^{2k+1}$. Let $p$ and $q$ be primes dividing $|G|$ such that $\Irr(B_p(G))$ contains no character of degree divisible by $q$. Then neither $p$ nor $q$ divides $r^2$ and:
\begin{itemize}
\item $q$ is odd;
\item $d_p(r^2)=d_q(r^2)$;
\item  a Sylow $q$-subgroup $Q$ of $G$ is abelian; and
\item there is a Sylow $p$-subgroup $P$ of $G$ such that $P$ normalizes $Q$.
\item  Further, either there is a Sylow $p$-subgroup $P$ of $G$ such that $[P,Q]=1$ or $G=\tw{2}\type{G}_2(r^2)$ with $p=2$ and $q\mid (r^2+1)$.
\end{itemize}
\end{lem}
\begin{proof}
Note that by Lemma \ref{lem:liedefining}, neither $p$ nor $q$ divides $r^2$, and as in the beginning of the proof of Lemma \ref{lem:exceptinitial}, we have $q\neq 2$.

First suppose that $G=\tw{2}\type{G}_2(3^{2k+1})$. If $p=2$, we see by \cite[p.~74]{ward} that there is a character of degree divisible by $q$ in $B_2(G)$ unless $q\mid (r^2+1)$. Note that in the latter situation, we have $Q$ is abelian and  we also have $4\mid (r^2+1)$ so that $d_2(3^{2k+1})=d_q(3^{2k+1})$. Further, from \cite[Sec.~2.1]{hiss91}, we see that $\norm{G}{Q}$ contains a Sylow $2$-subgroup in this case. Now suppose that both $p$ and $q$ are larger than 3. Then both $P$ and $Q$ are abelian  and we see from \cite[Sec.~4.1]{hiss91} that there is $\chi\in\Irr(B_p(G))$ with $q\mid\chi(1)$ whenever $d_p\neq d_q$. So $d_p=d_q$ and $[P,Q]=1$ as both can be chosen to lie in the same cyclic torus.

Next, let $G=\tw{2}\type{B}_2(2^{2k+1})$. Then both $p$ and $q$ are at least $5$ and the Sylow $p$ and $q$-subgroups are cyclic. If $d_p(r^2)\neq d_q(r^2)$, we see from \cite{burkhardt} that $B_p(G)$ contains a character of degree divisible by $q$. Hence we again have $d_p=d_q$ and $P$ and $Q$ can be chosen to lie in the same cyclic torus, giving $[P,Q]=1$.

Finally, let $G=\tw{2}\type{F}_4(2^{2k+1})$. Then we see using \cite[Bem.~1]{malle91} that $B_p(G)$ contains a character of degree divisible by $q$ except possibly when  $d_p(r^2)=d_q(r^2)$ and neither $p$ nor $q$ is 3. In that situation, there again exist $P$ and $Q$ in the same  torus and $[P,Q]=1$.
\end{proof}

\subsection{Classical Groups}

In the next several statements, we turn our attention to classical groups, aiming to obtain a statement similar to those for exceptional groups above.
Throughout, we will continue to let $p\neq q$ be two primes and let $r$ be a power of some prime $r_0$. As usual, $\GL_n(\epsilon r)$ will mean $\GL_n(r)$ if $\epsilon=1$ and $\operatorname{GU}_n(r)$ if $\epsilon=-1$, and similar for related groups like $\SL_n(\epsilon r)$ and their projective versions. 

As before, our strategy will often be to exhibit a unipotent character in the principal $p$-block whose degree is divisible by $q$. The next lemma covers an initial case.

\begin{lem}\label{lem:unipsTypeAp23}
Let $\wt{G}:=\GL_n(\epsilon r)$ with $n\geq 3$ and let $p\neq  q$ be primes not dividing $r$ but dividing $|\wt{G}|$. Let $e_p:=d_p(\epsilon r)$ and $e_q:=d_q(\epsilon r)$. Assume that $e_p\in\{1,2\}$ and that $e_q>e_p $ but $(n, e_q)\not\in\{(3,3), (4, 2)\}$. Then $B_p(\wt{G})$ contains a unipotent character of degree divisible by $q$.
\end{lem}
\begin{proof}
The unipotent characters of $\wt{G}$ correspond to partitions of $n$. If $p=2$, all unipotent characters lie in $B_p(\wt{G})$ by \cite[Thm.~21.14]{CE04}. Otherwise, by \cite{FS82} the unipotent characters in $B_p(\wt{G})$ are those whose $e_p$-cores are the same as that of the trivial character, which corresponds to the partition $(n)$. Namely, $B_p(\wt{G})$ contains all unipotent characters when $e_p=1$ and contains those with $2$-core $(n\mod 2)$ if $e_p=2$.

First, assume $e_q\geq 3$. Then it therefore suffices to show that there is a unipotent character of degree divisible by $q$ corresponding to a partition with  $2$-core $(n\mod 2)$.

Let $d_q:=d_q(r)$. Now, by \cite[Cor. 6.6]{malle07}, a unipotent character has degree divisible by $q$ if it does not lie in a $d_q$-Harish-Chandra series  $(L, \lambda)$ with $L=\cent{\wt{G}}{\mathbb{S}_{d_q}}$ the centralizer of a Sylow $d_q$-torus and $\lambda$ a unipotent character of $L$.  Further, letting $n=we_q+m$ with $0\leq m<e_q$, we have any unipotent character lying in such a $d_q$-Harish-Chandra series would correspond to a partition of $n$ with $e_q$-core  of size $m$, using \cite[Cors. 4.6.5 and 4.6.7]{GM20}. Hence, it suffices to exhibit a partition of $n$ with $e_q$-core not of size $m$ and $2$-core  $(n\mod 2)$.

If $m\not\in\{0, e_q-1\}$, then the partition $\tau:=(1^{m+1}, e_q-1)$ is an $e_q$-core of size $e_q+m>m$. The partition $\lambda_1:=(1^{e_q(w-1)+m+1}, e_q-1)$ then has $e_q$-core $\tau$ and has $2$-core $(n\mod 2)$ if $e_q$ is even. If $e_q$ and $m$ are both odd, then the partition $\lambda_2:=(1^{m+1}, e_q-1+e_q(w-1))$ instead has $e_q$-core $\tau$ and has $2$-core $(n\mod 2)$. Note that if $e_q=m+2$ or if $m=1$, then one of $\lambda_1$ or $\lambda_2$ satisfies our statement.

Now suppose $m\not\in\{1, e_q-2\}$ (and now including the cases $m=0$ or $e_q-1$). Then the partition $\mu:=(1^{m}, 2, e_q-2)$ is an $e_q$-core for $e_q\geq 4$, again of size $e_q+m$. If $e_q\geq 4$ is even or $m$ is even with $e_q\geq 5$ odd, then $\lambda_3:=(1^{m}, 2, e_q-2+e_q(w-1))$ has $e_q$-core $\mu$ and $2$-core $(n\mod 2)$.

We are therefore left with the case $e_q=3$ and $m\in\{0,2\}$. If $m=2$ and $e_q=3$, the partition $(1, 1, n-2)$ has $3$-core $(1, 1, 3)$ and $2$-core $(n\mod 2)$. So, assume $m=0$ and $e_q=3$, so $3\mid n$.  Our assumption  $(n, e_q)\neq (3,3)$ means that $n\geq 6$, so the partition $(2, n-2)$ has $3$-core $(2, 4)$ and $2$-core $(n\mod 2)$. This completes the proof when $e_q\geq 3$.

Finally, assume $e_p=1$ and $e_q=2$. From the discussions above, we have $B_p(\wt{G})$ contains every unipotent character and it suffices to know that there is a partition of $n$ with $2$-core different from $(n\mod 2)$.  If $n$ is odd, the partition $(1, n-1)$ has $2$-core $(1,2)$, satisfying this condition. If $n\geq 6$ is even, the partition $(1,2,n-3)$ has $2$-core $(1,2,3)$, again satisfying our condition. Given our assumption $(n,e_q)\neq (4,2)$, this completes the proof.
\end{proof}

In the following, let $r$ be a power of a prime $r_0$ and let $\wt{H}$ be one of $\GL_n(\epsilon r)$ with $n\geq 5$, $\Sp_{2n}(r)$ with $n\geq 2$, $\SO_{2n+1}(r)$ with $n\geq 3$,  or $\SO^\pm_{2n}(r)$ with $n\geq 4$. Assume that both $p$ and $q$ divide $|\wt{H}|$ and that $r_0\not\in\{p,q\}$. Let $d_p:=d_p(r)$ and $d_q:=d_q(r)$ and let $e_p$ denote $d_p(\epsilon r)$ if  $\wt{H}=\GL_n(\epsilon r)$ and $e_p:=d_p(r^2)$ otherwise, and similar for $e_q$. In the latter cases, let $\epsilon_p$ and $\epsilon_q$ be the numbers in $\{\pm 1\}$ such that $p\mid (r^{e_p}-\epsilon_p)$ and $q\mid (r^{e_q}-\epsilon_q)$.
The next observation extends the last. (Note that the cases $n\in\{3,4\}$ for $\GL_n(\epsilon r)$ are covered in the previous lemma.)

\begin{lem}\label{lem:unipsclassical}

Keep the hypotheses and notation in the preceeding paragraph. Assume that $e_p\mid e_q$  but that $e_p\neq e_q$. Further assume that $\epsilon_p=\epsilon_q$ (when relevant) and that $p\mid (r^{e_q}-\epsilon_q)$. Then $B_p(\wt{H})$ contains a unipotent character of degree divisible by $q$.
\end{lem}
\begin{proof}
Note that by Lemma \ref{lem:unipsTypeAp23}, we may assume that $e_p\geq 3$ in the case that $\wt{H}=\GL_n(\epsilon r)$ with $n\geq 5$.
Throughout, let $n=we_q+m$ with $0\leq m<e_q$, and let $m=w_1e_p+m_1$ with $0\leq m_1<e_p$.

The unipotent characters of $\wt{H}$ are labeled by partitions of $n$ if $\wt{H}=\GL_n(\epsilon r)$ and by certain symbols $X\choose Y$ in the other cases (see e.g. \cite[Section 13.8]{carter}). If $p=2$, then every unipotent character lies in $B_p(\wt{H})$, by \cite[Thm.~21.14]{CE04}. If $p$ is odd and $\wt{H}=\GL_n(\epsilon r)$, then two unipotent characters lie in the same $p$-block of $\wt{H}$ if they correspond to partitions that have the same $e_p$-core. In the other cases, if $p$ is odd then two unipotent characters lie in the same $p$-block of $\wt{H}$ if they correspond to symbols with the same $e_p$-core, respectively $e_p$-cocore,  if $\epsilon_p=1$, respectively $\epsilon_p=-1$ (see \cite[Thm.~(7A)]{FS82} and \cite[(12A) and (13B)]{FS89}). Hence the unipotent characters in  $B_p(\GL_n(\epsilon r))$ are those  whose corresponding partition has $e_p$-core $(m_1)$. For $\wt{H}=\Sp_{2n}(r)$ or $\SO_{2n+1}(r)$, the unipotent characters in $B_p(\wt{H})$ are those whose symbol has $e_p$-core (resp. $e_p$-cocore) $m_1 \choose \emptyset$. In $\SO^+_{2n}(r)$, the unipotent characters in $B_p(\wt{H})$ have $e_p$-core (resp. $e_p$-cocore) the same as that of $ n \choose 0$ and in $\SO_{2n}^-(r)$, they have $e_p$-core (resp. $e_p$-cocore) the same as that of ${0, n}\choose \emptyset$. In all cases, we will simply say that a partition or symbol has \emph{trivial $e_p$-core} (resp. \emph{trivial $e_p$-cocore}) if this is the case.

On the other hand, we again have by \cite[Cor. 6.6]{malle07}, that any unipotent character not lying in a $d_q$-Harish-Chandra series  $(L, \lambda)$ with $L$ the centralizer in $\wt{H}$ of a Sylow $d_q$-torus of the underlying algebraic group must necessarily have degree divisible by $q$.  Using \cite[Cors. 4.6.5, 4.6.7, and 4.6.16]{GM20}, the characters in such a series would have $e_q$-core (resp. $e_q$-cocore) of rank $m$.

In Table \ref{tab:unipcores}, we illustrate in each case an $e_q$-core $\mu$ of rank $m+e_q$ that further has trivial $e_p$-core. (Recall here that $e_p> 2$ in the $\GL_n(\epsilon r)$ case.) If $\mu= (\mu_1,\ldots, \mu_t)$ is a partition, one (or both) of either the partition $(1^{e_q(w-1)}, \mu)$ or the partition $(\mu_1, \ldots, \mu_{t-1}, \mu_t+e_q(w-1))$ then has trivial $e_p$-core and $e_q$-core $\mu$.
 If $\mu= {{X}\choose Y}$ is a symbol with $X=x_1, \ldots, x_t$, then the symbol $X' \choose Y$ with $X'=x_1, \ldots, x_{t-1}, x_{t}+e_q(w-1)$ has trivial $e_p$-core and has $e_q$-core $\mu$ (so in particular, of rank not $m$).
  Hence if $\wt{H}=\GL_n(\epsilon r)$ or if $\epsilon_p=1=\epsilon_q$, we have shown that there is a unipotent character of degree divisible by $q$ lying in $B_p(\wt{H})$.

Now suppose $\epsilon_p=-1=\epsilon_q$ with $\wt{H}$ one of the symplectic or special orthogonal groups and $p\mid (r^{e_q}+1)$. Note that in this situation, $\frac{e_q}{e_p}$ must be odd. Let $\mu$ be the symbol in Table \ref{tab:unipcores}. Then Olsson's theory of $e_p$-twists (see \cite[p. 235]{olsson}) yields a symbol $\bar\mu$ (the ``$e_p$-twist of $\mu$") with rank $m+e_q$ and with trivial $e_p$-cocore. Further, in each case we can see that our $\bar\mu$ is an $e_q$-cocore, using that $e_q/e_p$ is odd and the definition of the $e_p$-twist. Now, write $\bar\mu={{x_1, \ldots, x_t}\choose {y_1,\ldots, y_s}}$ with (without loss of generality) $x_t\geq y_s$. Then  we see that if $w-1$ is even, then the symbol $S:={{x_1, \ldots, x_{t-1}, x_t+e_q(w-1)}\choose {y_1,\ldots, y_s}}$ has rank $n$, $e_q$-cocore $\bar\mu$, and has trivial $e_p$-cocore since it can be converted to $\bar\mu$  by removing  $e_p$-cohooks.  If $w-1$ is instead odd, then  since $e_q/e_p$ is odd, the same properties hold  for $S':={{x_1, \ldots, x_{t-1}}\choose {y_1,\ldots, y_s, x_t+e_q(w-1)}}$. We remark that in the latter case, if $\wt{H}=\SO_{2n}^\epsilon(r)$, then $\bar\mu$ corresponds to a unipotent character of $\SO_{2(m+e_q)}^{-\epsilon}(r)$, which is also the type of the $d_q$-split Levi subgroup corresponding to the symbol $S'$. Hence in all cases, we have exhibited a unipotent character in $B_p(\wt{H})$ of degree divisible by $q$.
\end{proof}

\begin{table}
\begin{tabular}{|c|c|c|}
\hline
& $\GL_n(\epsilon r)$ & $\SO_{2n+1}(r)$ or $\Sp_{2n}(r)$  \\
\hline
$m=0$ & $(e_p, e_q-e_p)$ & ${0, e_q-e_p+1}\choose e_p$ \\
 & $e_p\neq 1$ &  \\
\hline
$m_1=0, m>0$& $(1^{m}, e_q)$ & ${1, 2, \ldots, m, e_q+m}\choose {0, 1, \ldots, m-1}$  \\
\hline
$m_1>0, w_1=0$& $(1^{e_q-e_p}, m+e_p)$ & ${1, 2, \ldots, e_q-e_p-1, e_q-e_p, e_q+m}\choose {0, 1, \ldots, e_q-e_p-1}$  \\
\hline
$m_1>0, w_1=1$& $(1^{e_p}, m_1+1, e_q-1)$ & ${1, 2, \ldots, e_q-1, e_q+e_p}\choose {0, 1, \ldots, e_q-1, m+e_q-e_p}$\\
\hline
$m_1>0, w_1\geq 2$& $(1^{e_q-1}, e_p+1, m-e_p)$ &${1, 2, \ldots, e_q-1, e_q+e_p}\choose {0, 1, \ldots, e_q-1, m+e_q-e_p}$ \\
\hline
\hline
\end{tabular}

\begin{tabular}{|c|c|c|}
\hline
&  $\SO^+_{2n}(r)$ & $\SO_{2n}^-(r)$ \\
\hline
$m=0$  & ${ e_q-e_p}\choose e_p$ & ${e_p, e_q-e_p}\choose\emptyset$\\
\hline
$m_1=0, m>0$ & ${1, 2, \ldots, m, e_q+m}\choose {0, 1, \ldots, m}$& ${0,1, 2, \ldots, m, e_q+m}\choose { 1, \ldots, m}$ \\
\hline
$m_1>0, w_1=0$ & ${1, 2, \ldots, e_q-e_p-1, e_q-e_p, e_q+m}\choose {0, 1, \ldots, e_q-e_p}$ & ${0,1, , \ldots, e_q-e_p-1, e_q-e_p, e_q+m}\choose { 1, \ldots, e_q-e_p}$ \\
\hline
$m_1>0, w_1=1$& ${1, 2, \ldots, e_q-1, m_1+e_q, e_q+e_p}\choose {0, 1, \ldots, e_q-1, e_q}$& ${1, 2, \ldots, e_q-1, e_q+e_p}\choose {0, 1, \ldots, e_q-1, e_q,  m+e_q-e_p}$\\
\hline
$m_1>0, w_1\geq 2$ & ${1, 2, \ldots, e_q-1, e_q+e_p, m+e_q-e_p}\choose {0, 1, \ldots, e_q-1, e_q}$& ${1, 2, \ldots, e_q-1, e_q+e_p}\choose {0, 1, \ldots, e_q-1, e_q,  m+e_q-e_p}$\\
\hline
\end{tabular}

\caption{\footnotesize Some $e_q$-cores of rank $e_q+m$ with trivial $e_p$-core when $e_p\mid e_q$, $e_p\neq e_q$}\label{tab:unipcores}

\end{table}

\begin{pro}\label{prop:Lieinitial}
Let $H=\bH^F$ be a quasisimple group of Lie type such that $\bH$ is a simple, simply connected algebraic group and $F\colon \bH\rightarrow\bH$ is a Frobenius endomorphism defining $H$ over $\FF_r$. Let $G$ be a group such that $G=H/Z$ for some $Z\leq \zent{H}$. Let $p$ and $q$ be two primes dividing $|G|$. Assume that  $B_p(G)$ contains no irreducible character of degree divisible by $q$. Then neither $p$ nor $q$ divides $r$ and:
\begin{enumerate}
\item $q$ is odd;
\item $d_p(r)=d_q(r)$;
\item  a Sylow $q$-subgroup $Q$ of $G$ is abelian; and
\item there is a Sylow $p$-subgroup $P$ of $G$ such that $P$ normalizes $Q$.
\item  Further, either there is a Sylow $p$-subgroup $P$ of $G$ such that $[P,Q]=1$, or $G=\PSL_p(\epsilon r)$ with $p\mid (r-\epsilon)$, or $(H, d_q(r), p)$ are as in (v) of Lemma \ref{lem:exceptinitial}.
\end{enumerate}

\end{pro}
\begin{proof}
First, note that if $H$ is of exceptional type, the statement follows from Lemma \ref{lem:exceptinitial}. We again have by Lemma \ref{lem:liedefining} that $r$ is a power of $r_0\not\in\{p,q\}$. 

Then we may assume $H$ is of classical type and may write $S:=H/\zent{H}$ so that $S=\PSL_n(\epsilon r)$ with $n\geq 2$ and $\epsilon\in\{\pm1\}$; $S=\PSp_{2n}(r)$ with $n\geq 2$, $S=\operatorname{P\Omega}_{2n+1}(r)$ with $n\geq 3$, or $S=\operatorname{P\Omega}_{2n}^\pm(r)$ with $n\geq 4$.  Let $(\bH^\ast, F)$ be dual to $(\bH, F)$ and write $H^\ast:=(\bH^\ast)^F$. Note that $[H^\ast, H^\ast]$ is also a simple group of Lie type by \cite[Prop.~24.21]{MT11}.

Let $s\in H^\ast$ with order a power of $p$.  We claim that a semisimple character $\chi_s\in\irr{H}$ corresponding to the $H^\ast$-class of $s$ lies in $B_p(H)$. Indeed, if $p\nmid |\zent{H}|$, then $p>2$ in the cases of the orthogonal and symplectic types, so we have $p$ is good for $H$,  $\cent{\bH^\ast}{s}$ is connected by \cite[Exer. 20.16]{MT11},  and $\chi_s\in\irr{B_p(H)}$ by \cite[Thm.~21.13]{CE04}. If instead $p\mid |\zent{H}|$, then either $p=2$ or $H=\SL_n(\epsilon r)$ and $p\mid (n, r-\epsilon)$. In either of these cases, $H$ has a unique unipotent block by  \cite[Thm.~21.14]{CE04} and \cite{FS82} and hence any semisimple character $\chi_s\in\irr{H}$ corresponding to the $H^\ast$-conjugacy class of $s$ lies in $B_p(H)$ by \cite[Thm.~9.12]{CE04}.

Now, if $s$ further lies in $[H^\ast, H^\ast]$, then $\chi_s$  is trivial on $\zent{H}$ by \cite[Lem.~4.4]{NT13}, and we can view $\chi_s\in\irr{B_p(G)}$ by deflation. (If $p\nmid |\zent{H}|= |H^\ast/ [H^\ast, H^\ast]|$, then this condition automatically holds.) In such a case,  since $\chi_s$ has degree $[H^\ast:\cent{H^\ast}{s}]_{r'}$, we see $s$ must centralize a Sylow $q$-subgroup of $H^\ast$.

By \cite[Thm. 2.2]{BFMMNSST}, we see that for any odd $p$, $[H^\ast, H^\ast]$ necessarily contains a $p$-element that does not centralize a Sylow $2$-subgroup of $[H^\ast, H^\ast]$, and hence does not centralize a Sylow $2$-subgroup of $H^\ast$. We  therefore see that $q$ is odd, giving (i).


\smallskip

Write $d_p:=d_p(r)$ and $d_q:=d_q(r)$. We will next show (ii), that $d_p=d_q$.   
Now, let $\wt{H}:=\GL_n( \epsilon r)$, $\Sp_{2n}(r)$, $\SO_{2n+1}(r)$, or $\SO_{2n}^\epsilon(r)$, respectively, in the cases $S=\PSL_n(\epsilon r)$, $\PSp_{2n}(r)$, $\operatorname{P\Omega}_{2n+1}(r)$or $\operatorname{P\Omega}_{2n}^\epsilon(r)$. Then $S$ can naturally be viewed as a normal subgroup of $\wt{H}/\zent{\wt{H}}$.

First let  $S=\PSL_n(\epsilon r)$ and write $e_q:=d_q(\epsilon r)$ and $e_p:=d_p(\epsilon r)$. We claim that $e_p\mid e_q$. If $e_p=1$, this is certainly true, so suppose $e_p\neq 1$, and hence $p\nmid (r-\epsilon)$. Then any semisimple $p$-element $\wt s\in [\wt{H}^\ast, \wt{H}^\ast]$ satisfies its image $s\in H^\ast$ has connected centralizer $\cent{\bH^\ast}{s}$ by \cite[Exer.~20.16]{MT11}, and therefore the corresponding semisimple character $\chi_{\wt{s}}\in\irr{B_p(\wt{H})}$ restricts irreducibly to $\irr{B_p(H)}$.
 Further, such a character is also trivial on $\zent{\wt{H}}$ by \cite[Lem.~4.4]{NT13}, so the restriction can be viewed in $\irr{B_p(G)}$. Then such a $\wt{s}$ must centralize a Sylow $q$-subgroup of $\wt{H}^\ast$, as otherwise $\chi_{\wt{s}}$ has degree divisible by $q$.
%
Then $\wt{s}$ centralizes a Sylow $d_q$-torus $\bbS_{d_q}$ of $\wt{H}^\ast$ by \cite[Thm. 5.9]{malle07}.

 Here $\cent{\wt H^\ast}{\bbS_{d_q}}\cong {T}\times \GL_m(\epsilon r)$ with $n=w e_q+m$,  $0\leq m<e_q$, and ${T}\cong\GL_1((\epsilon r)^{e_q})^w$ a torus centralizing  a Sylow $d_q$-torus of $\GL_{we_q}(\epsilon r)$ (see \cite[Ex.~3.5.14]{GM20}).
If  $p\mid |\GL_m(\epsilon r)|$ but $p\nmid |{T}|$, note that  $e_p\leq m<e_q$ and that any element of $p$-power order has at most $m$ nontrivial eigenvalues. But writing $m=w_1e_p+ m_1$ with $m_1<e_p$, we can find an element $s_1$ of $\GL_m(\epsilon r)$ with $w_1e_p$ nontrivial eigenvalues of $p$-power order. Since $2m<n$, we can embed the element $\mathrm{diag}(s_1, s_1^{-1}, I)$ into $\GL_n(\epsilon r)$, and this element has $2w_1e_p>m$ nontrivial eigenvalues, a contradiction. This forces $p\mid |{T}|$ after all, and hence $e_p\mid e_q$.

If $n\geq 5$, then by Lemmas \ref{lem:unipsTypeAp23} and \ref{lem:unipsclassical},  we have $e_p=e_q$, since unipotent characters are trivial on $\zent{\wt{H}}$ and restrict irreducibly to $S$. But this forces also $d_p=d_q$.

Now suppose $H=\SL_2( r)$. If $d_p\neq d_q$, then $d_p=1$ and $d_q=2$. Any $p$-element $s\in [H^\ast, H^\ast]$ yields a character $\chi_s\in\irr{B_p(G)}$ of degree $(r+1)/2$ or $(r+1)$, and hence  divisible by $q$, a contradiction.  

Next let $H=\SL_3(\epsilon r)$. If $e_q\neq 3$, we may again conclude using Lemma \ref{lem:unipsTypeAp23}. So assume $e_q=3$ and $e_p=1$. Here we see that  $B_p(H)$ is the unique unipotent $p$-block, and  that for any nontrivial $p$-element $s\in [H^\ast, H^\ast]$ with $|s|\mid (r-\epsilon)_p$, the corresponding semisimple character $\chi_s\in\irr{H}$ has degree divisible by $(r^2+\epsilon r+1)$. Then  any such character deflates to one in $B_p(G)$ with degree divisible by $q$, a contradiction.  

Now let $H=\SL_4(\epsilon r)$. Here we may assume by  Lemma \ref{lem:unipsTypeAp23} that $e_q=2$ and $e_p=1$. Again $B_p(H)$ is the unique unipotent $p$-block, and the image in $[H^\ast, H^\ast]$ of the $p$-element $\mathrm{diag}(\zeta, \zeta^{-1}, 1, 1)$ of $\wt{H}^\ast$, where $|\zeta|=(r-\epsilon)_p$ gives a semisimple character $\chi_s\in\irr{B_p(H)}$ trivial on the center and with degree divisible by $r+\epsilon$, and hence $q$, again giving a contradiction.



Now consider the remaining classical types. Since $q\neq 2$, again $B_p(\wt{H})$ does not contain a  character trivial on the center with degree divisible by $q$, as this would again yield a character of $B_p(G)$ with degree divisible by $q$. Here let $e_p:=d_p(r^2)$ and as before, let $\epsilon_p\in\{\pm1\}$ be such that $p\mid (r^{e_p}-\epsilon_p)$, and similar for $e_q, \epsilon_q$. Again write $n=we_q+m$ with $0\leq m<e_q$. The structure of $\cent{\wt H^\ast}{\bbS_{d_q}}$ is described in \cite[Ex.~3.5.15]{GM20}. 
We have $\cent{\wt H^\ast}{\bbS_{d_q}}\cong {T}\times H_m$ where $H_m$ is of classical type of rank $m$ and $|{T}|=(r^{e_q}-1)^w$ if $d_q=e_q$ is odd and $(r^{e_q}+1)^w$ if $d_q=2e_q$ is even. 
By similar arguments as above, we have $p\mid (r^{e_q}-1)$, resp. $p\mid (r^{e_q}+1)$ and this forces $e_p\mid e_q$, $\epsilon_p=\epsilon_q$, and $p\mid (r^{e_q}+1)$ when $\epsilon_p=-1$. Then we are in the situation of Lemma \ref{lem:unipsclassical}, and again $e_p=e_q$. Then $d_p=d_q$ since $d_p=e_p=e_q=d_q$ if $d_q$ is odd, and $d_p=2e_p=2e_q=d_q$ if $d_q$ is even. Hence we have shown (ii).

\smallskip

We next turn to (iii), the claim that $Q\in\Syl_q(G)$ is abelian. Assume even that $\hat Q\in\Syl_q(H)$ is nonabelian.  Then by \cite[Lems.~2.8--2.10]{KM17} and since $d_p=d_q$, we have $B_p(H)$ contains a unipotent character of degree divisible by $q$, unless $q=3$, $H=\SL_3(\epsilon r)$, and $e_p=e_3=1$. Then assume we are in the latter situation. Here again $B_p(H)$ is the unique unipotent $p$-block, and similar to the proof of (ii), we see that for any $p$-element $s\in [H^\ast, H^\ast]$, we have $\chi_s(1)$ is divisible by $3$. Hence any such character deflates to one in $B_p(G)$ with degree divisible by $3$, a contradiction, yielding that $\hat Q$, and thus $Q$, must be abelian as claimed.

\smallskip

The argument for (iv) is exactly as in Lemma \ref{lem:exceptinitial}. Let $\hat P\in\Syl_p(H)$ and $\hat Q\in\Syl_q(H)$. If $p$ is odd, then by \cite[Thm.~4.10.2]{GLS}, we have $\hat P$ is conjugate to a group of the form $P_T\rtimes P_W$, where $P_T$ lies in a Sylow $d$-torus of $H$ and $P_W$ normalizes that torus. But similarly, $\hat Q=Q_T$ lies in a Sylow $d$-torus. Then these may be chosen such that $[P_T, Q_T]= 1$ and $P_W$ normalizes $Q_T$. If $p=2$, then we still have $P_T\lhd \hat P$ and  $[P_T, Q_T]= 1$, and $\hat P/P_T$ still normalizes the Sylow $d$-torus. This yields (iv).

\smallskip

Finally, we prove (v).  Assume that no $P\in\Syl_p(G)$ and $Q\in\Syl_q(G)$ exist satisfying $[P,Q]=1$. From above, note that we may assume that  $\hat P\in\Syl_p(H)$  is nonabelian, as otherwise $\hat P=P_T$.

 Let $d:=d_p=d_q$ and $e:=e_p=e_q$ and write $n=ew+m$ as before, so that  $p\leq w<q$ since $\hat Q$ is abelian but $\hat P$ is not. Then there exists a $p$-element $ s\in {H}^\ast$  such that $|s|$ does not divide the exponent of $|\mathbb{T}|$, and hence cannot centralize a Sylow $q$-subgroup of ${H}^\ast$. Then the corresponding semisimple character $\chi_{{s}}$ of ${H}$ has degree divisible by $q$ and lies in $B_p({H})$.

If such an $s$ can be chosen in $[H^\ast, H^\ast]$, then $\chi_s$ is trivial on $\zent{H}$ and we obtain a contradiction. So, assume that this is not the case. Then $p\mid |H^\ast/[H^\ast, H^\ast]|= |\zent{H}|$. Then either $p=2$ or $p\mid (n,r-\epsilon)$ with $S=\PSL_n(\epsilon r)$. In the latter case when $p$ is odd, note that $e=1$, and assume that $2p\leq n$. 
Consider $\zeta\in\FF_{r^{2p}}^\times$ with order $(r^{p}-\epsilon)_p>(r-\epsilon)_p$. Then an element $\wt s\in\wt H^\ast$ with nontrivial eigenvalues $\{\zeta, \zeta^r, \zeta^{r^2}, \cdots, \zeta^{r^{p-1}}, \zeta^{-1}, \zeta^{-r},\cdots,\zeta^{-r^{p-1}}\}$ lies in $[\wt H^\ast, \wt H^\ast]$, giving a character $\chi_{\wt s}\in\irr{B_p(\wt{H})}$ trivial on $\zent{\wt H}$ with degree divisible by $q$, using the same considerations as above. Note that since $q>n$ here, we have $q\nmid [\wt{S}:S]$, where  $\wt{S}=\wt{H}/\zent{\wt{H}}$. Then the constituents of $\chi_{\wt{s}}$ on $H$ also have degree divisible by $q$, lie in $B_p(H)$, and are trivial on $\zent{H}$. This is a contradiction, so we must have $n=p$. This gives one of the listed groups in (v).

Now assume $p=2$. Since $H$ is a classical group, there is a unique unipotent $2$-block  by \cite[Thm.~21.14]{CE04}, so the characters indexed by any  $2$-element of $H^\ast$ lie in $B_2(H)$.
Let $\zeta\in\FF_{r^2}$ with $|\zeta|=(r^2-1)_2$. If $S=\PSL_n(\epsilon r)$, we may argue exactly as in the previous paragraph, now using  an element $\wt s\in\wt{H}^\ast$ with nontrivial eigenvalues $\{\zeta, \zeta^r,\zeta^{-1},\zeta^{-r}\}$, to see that $n=2$.  If $S=\PSp_4(r)$, $\PSp_6(r)$, or $\operatorname{P\Omega}_7(r)$,
then since $q\mid (r^2-1)$, we see explicitly that there is a unipotent character (which therefore lies in $B_2(S)$) of degree divisible by $q$.

Now suppose $n\geq 4$ and $S=\operatorname{P\Omega}_{2n}^\epsilon(r)$,  $\operatorname{P\Omega}_{2n+1}(r)$, or $\PSp_{2n}(r)$. Note that $e=1$.  We have $\GL_3(r)$ embeds naturally into $\wt{H}^\ast$, so that $\SL_3(r)$ embeds into $[\wt{H}^\ast, \wt{H}^\ast]$. An element of $\SL_3(r)$ with eigenvalues $\{\zeta, \zeta^r, \zeta^{-r-1}\}$ and $|\zeta|=(r^2-1)_2$ then yields a $2$-element $\wt s\in [\wt H^\ast, \wt H^\ast]$, giving a character $\chi_{\wt s}\in\irr{B_2(\wt{H})}$ trivial on $\zent{\wt{H}}$. Since $|\wt s|=(r^2-1)_2$ again does not divide the exponent of $|\mathbb{T}|$, it cannot centralize a Sylow $q$-subgroup.
Hence, we have obtained  (v), completing the proof.
\end{proof}

\subsection{Theorem \ref{thm:pnilphall} for Almost Simple Groups}

Here, we finally prove Theorem \ref{thm:pnilphall} for certain almost simple groups. For $S$ a simple group of Lie type defined in characteristic $r_0$, we have $\Aut(S)=\wt{S}\rtimes D$ for an appropriate group $D$ of graph-field automorphisms and $\wt{S}$ the group of inner-diagonal automorphisms of $S$. (See, e.g. \cite[Thms.~2.5.12 and 2.5.14]{GLS}.) 

We start with an elementary lemma  that will also be used in the proofs of Theorems \ref{thm:normalsylow} and \ref{thm:pnilphall} in the coming sections.

\begin{lem}
\label{step1}
Let $p\neq q$ be two  primes. Let $G$ be a finite group and assume that
$q$ does not divide $\chi(1)$ for every $\chi\in\Irr(B_p(G))$. If $N$ is subnormal in $G$, then $q$ does not divide $\theta(1)$ for every $\chi\in\Irr(B_p(N))$.
\end{lem}

\begin{proof}
Suppose first that $N$ is normal in $G$.
Let $\theta\in\Irr(B_p(N))$. Since $B_p(N)$ is covered by $B_p(G)$, \cite[Thm.~9.4]{nav} implies that there exists $\chi\in\Irr(G|\theta)$ in the principal $p$-block. Therefore $\chi$ has $q'$-degree and hence the same holds for $\theta$. The general case follows by induction on the subnormal length of $N$ in $G$.
\end{proof}

\begin{thm}
\label{lie}
Let $p$ and $q$ be two different primes.
Let $G$ be an almost simple group with socle $S$ a simple group of Lie type. Suppose that $G/S$ is a $q$-group and that $B_p(G)$ contains no irreducible character of degree divisible by $q$. Then there is a Sylow $p$-subgroup $\wt P$ of $G$ and a Sylow $q$-subgroup $\wt Q$ of $G$ such that $\wt P$ normalizes $ \wt Q$.
\end{thm}

\begin{proof} 
Note that we may assume $p\mid |S|$, as otherwise a Sylow $p$-subgroup of $G$ is trivial and the result holds. Assume that no irreducible characters in the principal $p$-block $B_p(G)$ have degree divisible by $q$.  Then every $\chi\in\irr{B_p(S)}$ also satisfies $q\nmid\chi(1)$ using Lemma \ref{step1}. Hence by Lemmas \ref{sporadic} and \ref{lem:liedefining}, we may assume
that $S\neq \tw{2}\type{F}_4(2)'$ and $p$ and $q$ are nondefining primes for $S$, with the possibility that $q\nmid |S|$. Let $r_0$ be the defining prime for $S$.

(I) First, assume that there is a Sylow $p$-subgroup $P$ of $S$ and Sylow $q$-subgroup $Q$ of $S$ such that $[P,Q]=1$. (Note that this is trivially true if $q\nmid |S|$.)
We claim that in this case, there exist Sylow $p$- and Sylow $q$- subgroups $\wt P$ and $\wt Q$ of $G$ such that $[\wt P, \wt Q]=1$. (Note that this follows from \cite[Thm. 1.2]{lwwz}, as   by \cite[Thm. 3.6]{mn}, this forces that $p\nmid \chi(1)$ for every $\chi\in\irr{B_q(S)}$ as well. Then since $p\nmid |G/S|$, we have further $p\nmid \chi(1)$ for every $\chi\in\irr{B_q(G)}$. However, we aim to provide a proof independent of the main result of \cite{lwwz}.)

We have $S=H/\zent{H}$ for $H=\bH^F$ with $\bH$ a simple, simply connected algebraic group defined in characteristic $r_0$ and $F\colon \bH\rightarrow\bH$ a Steinberg endomorphism.
Let $r=r_0^f$ such that $H$ is defined over $\FF_r$ in the case $F$ is Frobenius, or let $r^2=r_0^{2n+1}$ in the case that $H$ is a Suzuki or Ree group.

First suppose that $q\mid |S|$.
By \cite[Lem.~3.1 and Prop.~3.5]{mn}, combined with \cite[Lem.~2.1 and Prop.~2.2]{malle14}, we have
 $d_p(r)=d_q(r)=:d$ (resp.  $d_p(r^2)=d_q(r^2)=:d$); $p$ and $q$ are odd, good for $\bH$,  larger than $3$ if $H$ is of type $\tw{3}\type{D}_4$, and  do not divide $|\zent{\bH}^F:(\zent{\bH}^\circ)^F|\cdot |\zent{\bH^\ast}^F:(\zent{\bH^\ast}^\circ)^F|$, where $(\bH^\ast, F)$ is dual to $(\bH, F)$;  and there exist $\hat P\in\Syl_p(H)$ and $\hat Q\in\Syl_q(H)$ that are abelian with $\hat P\hat Q\leq \mathbb{S}_d$ for an ($F$-stable) Sylow $d$-torus $\mathbb{S}_d$ of $\bH$.

With these conditions, we also see that  $G/S$ is induced by field automorphisms. Similarly, if $q\nmid |S|$, we see that the only option is that $G/S$ is induced by field automorphisms. Say $G=S\langle\varphi\rangle$ for some field automorphism $\varphi$ of order a power of $q$.  Namely, we may let $\varphi\in\langle F_{r_0}^m\rangle$, where the group of field automorphisms has order $q^a\cdot m$ with $q\nmid m$ and $F_{r_0}$ denotes a generating field automorphism induced by a standard Frobenius. 

Let $s\in [H^\ast, H^\ast]$ be a $p$-element. As in the proofs of Lemma \ref{lem:exceptinitial} and Proposition \ref{prop:Lieinitial}, we obtain that the corresponding semisimple character $\chi_s$ lies in $\Irr(B_p(H))$, with its deflation lying in $B_p(S)$, and $s$ must centralize a Sylow $q$-subgroup of $H^\ast$. Then $q$ does not divide $[H^\ast: \cent{H^\ast}{s}]$. Further, $\chi_s$  is stable under $\varphi$, as otherwise $B_p(G)$ would contain some character lying above $\chi_s$ with degree divisible by $q$.

This means that the $H^\ast$-conjugacy class of $s$ is $\varphi$-stable. Since the number of $H^\ast$-conjugates of $s$ is prime to $q$, we see  there is even some $H^\ast$-conjugate that is invariant under $\varphi$. 
Without loss, we may assume that this is $s$ itself. That is, we have $s\in(\bH^\ast)^\varphi$, and therefore we see $p\mid (r_0^{dm}-1)$. 
By comparing the order polynomials, we see this means $\hat P$ can be chosen to be a Sylow $p$-subgroup of the corresponding group $H(r_0^m)$ of the same type defined over $\FF_{r_0^m}$. Then $\hat P$ is centralized by $\varphi$. If $q\nmid |S|$, this proves the claim. Otherwise, by \cite[Prop.~5.13]{malle07}, it follows that $\mathbb{S}_d$ (and hence $\hat Q$) is $\varphi$-stable. Then we can choose $P$ and $Q$ such that  $[P, Q\langle\varphi\rangle]=1$, again completing the claim.

(II) We now may assume that no such $P$ and $Q$ exist satisfying $[P,Q]=1$. Then by Lemma \ref{lem:suzree} and Proposition \ref{prop:Lieinitial}, we still have $q$ is odd, $d_p(r)=d_q(r)$ (resp. $d_p(r^2)=d_q(r^2)$), and there is a Sylow $p$-subgroup $P$ of $G$ (and hence of $G$) that normalizes an abelian Sylow $q$-subgroup $Q$ of $S$. Further,  either $S=\PSL_p(\epsilon r)$ with $p\mid(r-\epsilon)$; $p=2$ and  $S=\tw{2}\type{G}_2(r^2)$; $S=\type{E}_6^\epsilon(r)$, $p=2$,  and $4$ and $q$ divide $(r+\epsilon)$; or $p=3$ and $S=\tw{3}\type{D}_4(r)$.    
%
%
In each case, note that $q>p$ and the only $q$-elements of $\Out(S)$ are field automorphisms. So, up to conjugation by inner automorphisms, we may again assume  $G\leq S\langle\varphi\rangle$ with $\varphi$ a field automorphism of $q$-power order. 
Let  $r=r_1^{q^a}$ (resp. $r^2=r_1^{q^a}$ in the case of the Ree groups) for some $a\geq 1$ and some power $r_1$ of $r_0$. Note that in each case, we can again choose $P$ to already be a Sylow $p$-subgroup of $\PSL_p(\epsilon r_1)$, resp. $\tw{2}\type{G}_2(r_1)$, $\type{E}_6^\epsilon(r_1)$, $\tw{3}\type{D}_4(r_1)$ for any such $r_1$. In particular, $P$ is centralized by $\varphi$. Then since $P$ normalizes $Q$, we see $P$ normalizes the Sylow $q$-subgroup $\wt{Q}=Q\langle\varphi\rangle$ of $G$ as well.
\end{proof}


\section{More on simple groups}

 We will also need the following  results on simple groups.

\begin{thm}
\label{as}
Let $G$ be an almost simple group with socle $S$. Let  $p$ be a prime divisor of $|S|$  and let  $q\neq p$ be another prime. Suppose that $G/S$ is a $p'$-group and that $B_p(G)$ contains no irreducible character of degree divisible by $q$. Then either $G/S$ has a normal Sylow $q$-subgroup or $S=\PSL_n(\epsilon r)$ with Sylow $p$ and $q$-subgroups of $S$ such that $[P,Q]=1$ and $p, q\nmid (r-\epsilon)$.

In the latter case, let $T:=\wt{S}\cap G$, and otherwise let $T:=S$. Then $S\leq T\lhd G$ satisfies  $|T/S|$ is coprime to $pq$  
and $G/T$ has a normal abelian Sylow $q$-subgroup.
\end{thm}

\begin{proof}
First, note that if $\Out(S)$ is abelian, then $G/S$ has a normal Sylow $q$-subgroup. So, we may assume that $S$ is not an alternating, sporadic, or Tits group. {The groups with exceptional Schur multipliers can further be checked in GAP \cite{gap}.}

Suppose for the rest of the proof that $S$ is a simple group of Lie type defined in characteristic $r_0$ with a nonexceptional Schur multiplier. Further, since $\Out(S)$ is cyclic for the Suzuki, Ree, and triality groups, we assume that $S$ is not one of these. As in the proof of Theorem \ref{lie}, we may assume by Lemma \ref{lem:liedefining}
that either $q\nmid |S|$ or that $r_0\not\in\{p,q\}$. We may find a quasi-simple group of Lie type $H$ such that  $S=H/\zent{H}$ and $H=\bH^F$, where $\bH$ is a simple algebraic group of simply connected type over $\overline{\mathbb{F}}_{r_0}$ and $F\colon \bH\rightarrow\bH$ is a Frobenius endomorphism endowing $H$ with an $\FF_{r}$-rational structure, where $r=r_0^f$ for some positive integer $f$.
Further, as there we obtain that $d_p=d_q$ if $q\mid |H|$ and either there exists $P\in\Syl_p(S)$ and $Q\in\Syl_q(S)$ such that $[P,Q]=1$, or $S=\PSL_p(\epsilon r)$ with $p, q\mid(r-\epsilon)$, or $S=\type{E}_6^\epsilon (r)$ with $p=2$ and $4,q\mid (r+\epsilon)$. In the latter cases, recall that there is $P\in\Syl_p(S)$ and $Q\in\Syl_q(S)$ with  $P$ normalizing $Q$.

  Let $\hat P$ and $\hat Q$ be Sylow $p$- and $q$-subgroups of $H$ such that $P=\hat P\zent{H}/\zent{H}$ and $Q=\hat Q\zent{H}/\zent{H}$. Now, recall $\Out(S)$ can be realized as $\wt{S}/S\rtimes D$, where $\wt{S}$ is the group of inner-diagonal automorphisms of $S$ and $D$ is an appropriate group of graph-field automorphisms. If $[P,Q]=1$ and   $q\mid |H|$, then using \cite[Lem. 3.1 and Prop. 3.5]{mn}, we have $[\hat P, \hat Q]=1$; $p$ and $q$ are odd;  $\hat P$ and $\hat Q$ are abelian; and  we have further from \cite[Prop. 2.2 and Lem. 2.1]{malle14}, that $p$ and $q$ are good for $\bH$ and do not divide $|\zent{H}|$.
The condition $q\nmid|H|$ or the above conditions on $q$ when $q\mid |H|$ yield that in either case, $q\nmid |\wt{S}/S|$. In the  cases when no such $P$ and $Q$ exist, we also see by our conditions on $p$ and $q$ that $q\nmid |\wt{S}/S|$.

Further, in either situation, $q$ does not divide the order of a graph automorphism in $D$ unless possibly if $\bH=\type{D}_4$ and $q=3$, in which case $\hat Q\neq 1$ is not abelian. Hence we see a Sylow $q$-subgroup $X$ of $\Out(S)$ is cyclic with order dividing $f$ and can be chosen to be generated by a field automorphism. Since the graph automorphisms in $D$ commute with the field automorphisms, we wish to show that $X$ is also normalized by $\wt{S}/S$.

Assume that $\bH$ is not of type $\type{A}_{n-1}$. Then $\wt{S}/S$ is either cyclic of size at most 3 or is Klein four with $\bH$ of type $\type{D}_n$, $n\geq 4$. In the latter case, we again note that $q\neq 3$ as above. Then in these cases, $q$ does not divide the order of the automorphism group of $\wt{S}/S$, and hence $X$ acts trivially on $\wt{S}/S$.

Finally, we consider the case that $\bH=\type{A}_{n-1}$, so that $H=\SL_n(\epsilon r)$ with $n\geq 2$ and $\epsilon\in\{\pm1\}$.  First, assume that $q\mid |H|$ and $q\mid(r-\epsilon)$. Then since
$Q$ is abelian, this forces $q>n$. If $q\nmid|H|$, we also have $q>n$.
 In particular, $|\aut{\wt{S}/S}|<n$ must be relatively prime to $q$. Then again we see $X$ must act trivially on $\wt{S}/S$. 

 Hence, we may assume that $q\mid |H|$ and $q\nmid(r-\epsilon)$. If $n=p=2$, then again $X$ acts trivially on $\wt{S}/S$. Then since $d_p(r)=d_q(r)$, we also have $p\nmid (r-\epsilon)$. Further, note that we are in the case $[P,Q]=1$.
 Here let $T:=\wt{S}\cap G\lhd G$. Then $|T/S|$ is  prime to both $p$ and $q$ and $G/T$ is abelian. The statement then follows.
%
\end{proof}

As a corollary to our work so far and that of \cite{mn}, we obtain Brauer's Height Zero Conjecture for two primes (\cite[Conj.~A]{mn}) for almost simple groups. (Note that this of course would  follow from \cite[Thm.~1.2]{lwwz}, but our goal is to provide an alternate proof in Theorem \ref{thm:2pBHZ} below.)
\begin{cor}\label{cor:2pBHZalmostsiple}
Let $A$ be an almost simple group $S\leq A\leq\Aut(S)$ with $S$ a nonabelian simple group. Suppose that $p\neq q$ are primes dividing $|A|$ such that $B_p(A)$ contains no character of degree divisible by $q$ and $B_q(A)$ contains no character of degree divisible by $p$. Then there are $\wt P\in\Syl_p(A)$ and $\wt Q\in\Syl_q(A)$ such that $[\wt P, \wt Q]=1$.
\end{cor}
\begin{proof}
Note that by Lemma \ref{step1}, $B_p(S)$ has no character of degree divisible by $q$ and $B_q(S)$ has no character of degree divisible by $p$. Then there is $P\in\Syl_p(S)$ and $Q\in\Syl_q(S)$ such that $[P,Q]=1$ by \cite[Thm.~5.1]{mn}. Note that we may therefore assume $A\neq S$. As in part (I) of the proof of Theorem \ref{lie}, we may assume that $S$ is a  simple group of Lie type defined in characteristic $r_0\not\in\{p,q\}$ and let $S=H/\zent{H}$ as in there. We again have
 $d_p(r)=d_q(r)=:d$ (resp.  $d_p(r^2)=d_q(r^2)=:d$) and the same constraints on $p$ and $q$ as there. 
In particular, we see that the Sylow $p$- and $q$-subgroups of $A/S$ are induced by field automorphisms. Recall that $\Aut(S)=\wt{S}\rtimes D$ with $D$ abelian unless $S=\type{D}_4(r)$, in which case $D\cong \sym_3\times C_f$. In each case, there is a normal subgroup $A_{p'}\lhd A$ such that $A/A_{p'}$ is a $p$-group and $A_{p'}/S$ is a $p'$-group. By Lemma \ref{step1}, we have $B_p(A_{p'})$ has no character of degree divisible by $q$ and $B_q(A_{p'})$ has no character of degree divisible by $p$, and therefore by Theorem \ref{as}, either $A_{p'}/S$ has a normal Sylow $q$-subgroup or $S=\PSL_n(\epsilon r)$ with $p$ and $q$ not dividing $r-\epsilon$. In the latter case, we have  $T=\wt{S}\cap A_{p'}=\wt{S}\cap A$ satisfies $T/S$ is a $\{p,q\}'$-group and $A_{p'}/T$ has a normal Sylow $q$-subgroup.
%
 Similarly, there is a normal subgroup $A_{q'}\lhd A$ such that $A/A_{q'}$ is a $q$-group, $A_{q'}/S$ is a $q'$-group, and $A_{q'}/T$ has a normal Sylow $p$-subgroup, where $T=\wt{S}\cap A$ in the exceptional case of $\PSL_n(\epsilon r)$ from Theorem \ref{as} or $T=S$ in the other cases.

Let $X_p$ and $X_q$ be subgroups of $A$ such that $X_p/T$ and $X_q/T$ are Sylow $p$- and $q$-subgroups of $A/T$ with  $X_q\lhd A_{p'}$ and $X_p\lhd A_{q'}$. Applying Lemma \ref{step1} to each $X\in\{X_p, X_q\}$, we have  $B_p(X)$ has no character of degree divisible by $q$ and $B_q(X)$ has no character of degree divisible by $p$. Then by part (I)  of the proof of Theorem \ref{lie} (note that the same proof works in the case $T\neq S$ here), there is $\wt{Q}\in\Syl_q(X_q)$ and $\wt{P}\in\Syl_p(X_p)$ (which are therefore Sylow $q$- and $p$-subgroups of $A$) such that $[\wt{Q}, P]=1=[Q, \wt{P}]$. Since $\wt{Q}=Q\langle\varphi\rangle$ and $\wt{P}=P\langle\psi\rangle$ with $\varphi, \psi$ (commuting) field automorphisms, we therefore have $[\wt{P}, \wt{Q}]=1$.
\end{proof}

\begin{thm}
\label{quasi}
Let $p\neq q$ be two primes. Let $G$ be a quasisimple group with $\bZ(G)>1$ a $p$-group. If $q$ divides $|G|$, then there exists $\chi\in\Irr(B_p(G))$ of degree a multiple of $q$.
\end{thm}
\begin{proof}

Let $H$ be the full Schur covering group of the simple group $S:=G/\zent{G}$. Then $G=H/Z$ for some $Z\leq\zent{H}$ and our assumption that $|\zent{G}|$ is a power of $p$ implies that $p$ divides $|\zent{H}|$. Note that $\irr{B_p(S)}\subseteq \irr{B_p(G)}$, viewed by inflation. {The cases that $S$ is sporadic}  or that $\zent{H}$ is an exceptional Schur multiplier can be checked in GAP \cite{gap}. 
 By Theorem \ref{alt}, we may assume that $S$ is not an alternating group. Then we may assume that $H$ is a group of Lie type and that $\zent{H}$ is a nonexceptional Schur multiplier for $S$. Further, by Lemma \ref{lem:liedefining}, we may assume that $H$ is defined in characteristic $r_0\not\in\{p,q\}$.
Since we have assumed $\zent{H}$ is nontrivial, $H$ is also not a Suzuki or Ree group. Hence $H$ is as in  Proposition \ref{prop:Lieinitial}. Since $|\zent{G}|$ is a power of $p$, $G$ is not one of the exceptions listed in Proposition \ref{prop:Lieinitial}(v), and therefore there is $P\in\Syl_p(G)$ and $Q\in\Syl_q(G)$ with $[P,Q]=1$.
 Then as in the proof of Theorem \ref{as}, we obtain  that  $p\nmid |\zent{H}|$, a contradiction.
\end{proof}

\section{Brauer's height zero conjecture for two primes}

As we have mentioned, the so-called Brauer's height zero conjecture for two primes has been proven in \cite{lwwz}. The proof uses $p^*$-theory \cite{will} and some specific results on simple groups. We present a proof with a more elementary reduction to (some of) the problems on simple groups that we have solved in this paper.

The following is the main theorem of \cite{lwwz}. Our arguments in this proof are part of the arguments that we will use in the proof of Theorem C. We refer the reader to Section 9A of \cite{isa2} for the elementary properties of the generalized Fitting subgroup $\bF^*(G)$ and the layer $\bE(G)$.

\begin{thm}[Liu-Wang-Willems-Zhang]\label{thm:2pBHZ}
Let $G$ be a finite group and let $p\neq q$ be two primes. If $p$ does not divide $\chi(1)$ for every $\chi\in\Irr(B_q(G))$ and $q$ does not divide $\chi(1)$ for every $\chi\in\Irr(B_p(G))$ then $G$ has a nilpotent Hall $\{p,q\}$-subgroup.
\end{thm}

\begin{proof}
First, we note that by Lemma \ref{step1}  the hypotheses are inherited by normal subgroups and factor groups. Let $G$ be a minimal counterexample. Let $\pi=\{p,q\}$. We may assume that $\bO^{\pi'}(G)=G$ and $\bO_{\pi'}(G)=1$. We claim that $G$ has a unique minimal normal subgroup. If $M$ and $N$ are two different normal subgroups, then $G$ is isomorphic to a subgroup of $G/M\times G/N$. By the minimality of $G$ as a counterexample, $G/M$ and $G/N$ have a nilpotent Hall $\{p,q\}$-subgroup.  By  \cite[Cor.~ 8]{vdo}, $G=G/(M\cap N)$ has a nilpotent Hall $\{p,q\}$-subgroup. The claim follows.

Let $N=S\times\cdots\times S$, with $S$ simple,  be the minimal normal subgroup of $G$. Suppose first that $S$ is abelian.
Then we may assume that $S$ is cyclic of order $p$. Suppose that $E=\bE(G)>1$, so that $G$ has a quasisimple subnormal subgroup $T$. Since $N$ is the unique minimal normal subgroup of $G$, this implies that $Z=\bZ(T)>1$ is a $p$-group.
Now, Theorem \ref{quasi} yields that there exists $\varphi\in\Irr(B_q(N))$ of degree divisible by $p$. This is a contradiction with Lemma \ref{step1}. Therefore, $E=1$ and $\bF^*(G)=\bO_p(G)$.
Hence, $\bC_G(\bO_p(G))\leq\bO_p(G)$ and $G$ has a unique $p$-block by \cite[Cor.~ V.3.11]{fei}. It follows that $q$ does not divide the degree of any irreducible character of $G$ and, by the It\^o--Michler theorem, $G$ has a normal abelian Sylow $q$-subgroup.
By the uniqueness of $N$, this implies that $q$ does not divide $|G|$, and the result is obvious.

Now, we may assume that $S$ is simple nonabelian. Let $n$ be the number of copies of $S$ that appear in $N$. If $n=1$, then $G$  is almost simple and the result follows from Corollary \ref{cor:2pBHZalmostsiple}.
Therefore, we may assume that $n>1$. Write $N=S_1\times\cdots\times S_n$, with $S_i\cong S$.
We have that $G$ is isomorphic to a subgroup of $\Gamma=\Aut(S)\wr\SSS_n$.
Let $B=\Aut(S)^n\cap G$ so that $G/B$ is a transitive permutation group on $\Omega=\{S_1,\dots,S_n\}$. Since $\bO^{\pi'}(G)=G$, we may assume that $p$ divides $|G/B|$.

Suppose first that $q$ divides $|S|$.
By \cite[Lem.~3.2]{dmn}, there exist $\Theta,\Delta\subseteq\Omega$ such that $p$ divides $|G:(G_{\Theta}\cap G_{\Delta})|$. Since $S$ is simple, the order of a Sylow $2$-subgroup of $S$ exceeds $2$. By \cite{bra}, $|\Irr(B_q(S))|\geq3$, so we may take $\gamma_1,\gamma_2,\gamma_3\in\Irr(B_q(S))$. If we choose $\delta\in\Irr(N)$ to be an appropriate product of copies of  $\gamma_1,\gamma_2$ and $\gamma_3$ (copies of $\gamma_1$ in the copies of $S$ corresponding to $\Theta$, $\gamma_2$ in the copies of $S$ corresponding to $\Delta$ and of $\gamma_3$ elsewhere), we have that $p$ divides $|G:G_{\delta}|$. Therefore, $p$ divides the degree of any irreducible character of $G$ over $\delta$. Since some of them belongs to the principal $q$-block, we have a contradiction. We have thus seen that if $p$ divides $|G/B|$ then $q$ does not divide $|N|$. Similarly, if $q$ divides $|G/B|$, then $p$ does not divide $|N|$. Since $\bO^{\pi'}(G)=G$ and $\bO_{\pi'}(G)=1$ this means that one, and only one, of the primes in $\pi$ divides $|G/B||N|$, and that prime divides both factors.

So suppose that $p$ is the prime in $\pi$ that divides $|G/B|$ and $|N|$.
By the inductive hypothesis, $G/N$ and $N$ have nilpotent Hall $\pi$-subgroups. By \cite[Thm.~D5]{hal}, this implies that $G$ satisfies $D_{\pi}$. Recall that this means that $H$ has Hall $\pi$-subgroups, all of them are conjugate  and any $\pi$-subgroup of $G$ is contained in some Hall $\pi$-subgroup. 

Let $H_0=P_0$ be a Sylow $p$-subgroup of $N$. Since $B$ has nilpotent Hall $\pi$-subgroups, \cite{wie} implies that $H_0\leq H_1=P_1\times Q_1$ for some $P_1\in\Syl_p(B), Q_1\in\Syl_q(B)$. Since $G$ is in $D_{\pi}$, there exists  a Hall $\pi$-subgroup $H$ of $G$ such that $H$ contains $H_1$.
Since $q$ does not divide $|G/B|$, $Q_1$ is a Sylow $q$-subgroup of $G$. Thus $H=PQ_1$ for some $P\in\Syl_p(G)$. Note that $H_1=H\cap B\trianglelefteq H$. By Frattini's argument, $H=H_1\bN_{H}(Q_1)$. Since $\bN_{H}(Q_1)$ contains $H_1$, it follows that $Q_1$ is normal in $H$. Thus $[Q_1,P]\leq Q_1$. Furthermore, $P_0=H\cap N$ is normal in $H$ and $H/P_0$ is nilpotent. Therefore, $[Q_1,P]\leq P_0$. It follows that $[Q_1,P]\leq Q_1\cap P_0=1$. Hence, $H=Q_1\times P$ is nilpotent, as wanted.
\end{proof}

\section{Proof of the main results}

In this section, we prove the main results.
We start with the following elementary group theoretical lemma.

\begin{lem}
\label{sub}
Let $\pi$ be a set of primes and let $p\in\pi$. Suppose that $N$ is a subnormal subgroup of a finite group $G$ and $G$ has a $p$-nilpotent Hall $\pi$-subgroup $H$. Then $H\cap N$ is a $p$-nilpotent Hall $\pi$-subgroup of $N$.
\end{lem}

\begin{proof}
First, we note that any subgroup of a $p$-nilpotent group is $p$-nilpotent so it suffices to prove that $H\cap N$ is a Hall $\pi$-subgroup of $N$. If $N\trianglelefteq G$ this follows from the facts that $N\cap H$ is a $\pi$-group and  $|N:(N\cap H)|=|NH:H|$ is a $\pi'$-number and the general case follows by induction on the subnormal length.
\end{proof}

We will use several times the following.

\begin{lem}
\label{copr}
 Let $p$ and $q$ be two different primes. Suppose that $G=AN$ is the semidirect product of a $q$-group $A$ acting on a $q'$-group $N$. If $q$ does not divide $\chi(1)$ for every $\chi\in\Irr(B_p(G))$ then every character in $\Irr(B_p(N))$ is $G$-invariant.
 \end{lem}
 
 \begin{proof}
 By way of contradiction, suppose that there exists $\theta\in\Irr(B_p(N))$ that is not $G$-invariant. If follows from Clifford's correspondence \cite[Thm. ~6.11]{isa} that $q$ divides the degree of any character in $\Irr(G|\theta)$. Since $B_p(G)$ covers $B_p(N)$, \cite[Thm. ~9.4]{nav} implies that there exists $\chi\in\Irr(G|\theta)$ in the principal $p$-block of $G$. This contradicts the hypothesis.
 \end{proof}

The following was proved for arbitrary blocks in \cite{nav04} assuming the Alperin--McKay conjecture. Here we will prove the case of principal blocks using Theorem \ref{thm:2pBHZ}.

\begin{thm}
\label{main}
Suppose that a finite group $A$ acts coprimely on $G$. Let $p$ be a prime and let $P$ be an $A$-invariant Sylow $p$-subgroup of $G$. If every $\chi\in\Irr(B_p(G))$ is $A$-invariant, then $[A,P]=1$.
\end{thm}

\begin{proof}
 We argue by induction on $|AG|$. It suffices to see that for  $a\in A$, $P$ commutes with $a$. Thus, we may assume that $A=\langle a\rangle$ is cyclic of $q$-power order for some prime $q$. Since $AG$ is $q$-solvable, it follows from \cite[Thm. ~10.20]{nav}  that $\Irr(B_q(AG))=\Irr(A)$ is a set of characters of $p'$-degree.

 On the other hand, let $\chi\in\Irr(B_p(AG))$. Let $\theta\in\Irr(G)$ lying under $\chi$. Since $\theta$ belongs to $B_p(G)$, it follows that $\theta$ is $A$-invariant, so $\chi=\lambda\theta$ for some $\lambda\in\Irr(AG/G)$. In particular, $\chi$ has $q'$-degree. Now, Theorem \ref{thm:2pBHZ} implies that $A$ centralizes $P$, as wanted.
 \end{proof}

The following is Theorem C, which we restate.

\begin{thm}
Let $p, q$ be two different primes. Let $G$ be a finite group and assume that $S$ is not a composition factor of $G$ if $(S,p,q)$ is one of the $5$ exceptions listed in Theorem \ref{sporadic}.
Suppose that $q$ does not divide $\chi(1)$ for every $\chi\in\Irr(B_p(G))$. Then $G$ has a $p$-nilpotent Hall $\{p,q\}$-subgroup.
\end{thm}

\begin{proof}

We argue by induction on $|G|$.
We split the proof in a series of steps.

{\it Step 1:}
We may assume that $\bO_{p'}(G)=1$ and $\bO^{q'}(G)=G$.

\medskip

Suppose that $K=\bO_{p'}(G)>1$. By the inductive hypothesis, there exists $P\in\Syl_p(G)$ such that $PK/K$ normalizes $QK/K$ for some $Q\in\Syl_q(G)$, i.e.,
$$
PK/K\leq\bN_{G/K}(QK/K)=\bN_G(Q)K/K.
$$
Hence, $P\leq \bN_G(Q)K$. Since $K$ is a $p'$-group, we conclude that $\bN_G(Q)$ contains a Sylow $p$-subgroup of $G$, as wanted.

Now, suppose that $N$ is a proper normal subgroup of $G$  of $q'$-index.  Let $P$ be a Sylow $p$-subgroup of $G$, so that $|PN|_{\{p,q\}}=|G|_{\{p,q\}}$. By Lemma \ref{step1} and the inductive hypothesis, $N$ has a $p$-nilpotent Hall $\{p,q\}$-subgroup. It follows from Lemma 2.1 of \cite{lwwz} that $G$ has  a Sylow $p$-subgroup that  normalizes a Sylow $q$-subgroup of $G$, as desired.

\medskip

Now, let $E=\bE(G)$ be the layer of $G$.

\medskip

{\it Step 2:} We may assume that $E>1$.

\medskip

Suppose that $E=1$. Then $\bF^*(G)=\bO_p(G)$ and by  \cite[Thm. ~ 9.8]{isa2} $\bC_G(\bO_p(G))\subseteq\bO_P(G)$. Now, \cite[Cor. ~ V.3.11]{fei} implies that $G$ has a unique $p$-block. Therefore, $q$ does not divide $\chi(1)$ for every $\chi\in\Irr(G)$ and by the It\^o--Michler theorem $G$ has a normal abelian Sylow $q$-subgroup.
The result follows.

\medskip

Let $Z=\bZ(E)$, so that $Z$ is a $p$-group (by Step 1) and $E/Z$ is a direct product of nonabelian minimal normal subgroups of $G/Z$ of order divisible by $p$.
Write $E/Z=N_1/Z\times\cdots\times N_t/Z$, where $N_i/Z$ are the non-abelian minimal normal subgroups of $G/Z$.  Let $N_i/Z$ be the product of $j_i$ copies $S_{j_1},\dots,S_{j_i}$ of a non-abelian simple group $S_i$.   Let $C_i/Z=\bC_{G/Z}(N_i/Z)$, so that $G/C_i$ is isomorphic to a subgroup of $\Aut(N_i/Z)=\Aut(S)\wr\SSS_{j_i}$
that contains $M_i/C_i\cong N_i/Z$.

  \medskip
{\it Step 3:}
We claim that $j_i=1$ for every $i=1,\dots,t$.

\medskip
Let $L_i/C_i=\Aut(S)^{j_i}\cap G/C_i$.
 Suppose that $q$ divides $|G/L_i|$.
Note that $\tilde{G}=G/L_i$ acts on $N_i/Z$ as a permutation group on $\Omega=\{S_{j_1},\dots,S_{j_i}\}$. By  \cite[Lem.~3.2]{dmn}, there exist $\Gamma,\Delta\subseteq \Omega$ such that $q$ divides $|\tilde{G}:(\tilde{G}_{\Gamma}\cap\tilde{G}_{\Delta})|$. It follows from  \cite[Thm.~3.18]{nav} and \cite{bra} that the principal $p$-block of $S_i$ has at least $3$ irreducible characters (because a Sylow $2$-subgroup of $S$ has order at least $4$). Let
$\gamma_1,\gamma_2,\gamma_3\in\Irr(B_p(S_i))$. Choosing $\psi\in\Irr(B_p(M_i/C_i))$ as a suitable product of copies of $\gamma_1,\gamma_2$ and $\gamma_3$, we can obtain that  $q$ divides $|G:G_\psi|$. It follows from Clifford's correspondence  \cite[Thm.~6.11]{isa} that for every $\chi\in\Irr(G|\psi)$, $q$ divides $\chi(1)$.  Using again \cite[Thm.~9.4]{nav}, we see that there exists $\chi\in\Irr(G|\psi)$ in $B_p(G)$. This contradicts the hypothesis.

Therefore, we may assume that $G/L_i$ is a $q'$-group. Since by Step 1 $G=\bO^{q'}(G)$, we deduce that $G=L_i$. This implies that $G/C_i$ is isomorphic to a subgroup of $\Aut(S)^{j_i}$. Since $M_i/C_i$ is a minimal normal subgroup of  $G/C_i$, this forces $j_i=1$.

\medskip

In other words, we have shown that the non-abelian minimal normal subgroups of $G/Z$ are simple groups. Write $E/Z=X_1\times\cdots\times X_s$ where $X_1,\dots, X_s$ are the non-abelian (simple) minimal normal subgroups of $G/Z$.
Let $C/Z=\bC_{G/Z}(E/Z)$, so that $G/C$ is isomorphic to a subgroup  of $\Aut(X_i)\times\cdots\times\Aut(X_s)$ that contains $EC/C\cong E/Z$. Put $M=EC $.

\medskip
{\it Step 4:}  $G/M$ contains a normal subgroup $T/M$ such that $G/T$ is an abelian $q$-group and $|T/M|$ is prime to $p$ and $q$.

\medskip

  Note that  $G/M$ is isomorphic to a subgroup of $\Out(X_1)\times\cdots\times\Out(X_s)$ and by Schreier's conjecture $G/M$ is solvable. Write $\overline{G}=G/M$. By  \cite[Thm.~10.20]{nav}, we have that $\Irr(\overline{G}/\bO_{p'}(\overline{G}))=\Irr(B_p(\overline{G}))\subseteq\Irr(B_p(G))$ is a set of characters of $q'$-degree. By  It\^o's theorem \cite[Thm.~12.33]{isa}
$\overline{G}/\bO_{p'}(\overline{G})$ has a normal Sylow $q$-subgroup. This implies that $q$ does not divide $\overline{G}/\bO_{p'}(\overline{G})$, so by Step 1 $\overline{G}=\bO_{p'}(\overline{G})$ is a $p'$-group.  

Recall that  $G/C$ is isomorphic to a subgroup  of $\Aut(X_1)\times\cdots\times\Aut(X_s)$ that contains $X_1\times\cdots\times X_s$. Let $\pi_i:G/C\longrightarrow \Aut(X_i)$ be the projection homomorphism, so that $\pi_i(G/C)$ is isomorphic to a factor group of $G$. Thus, $\Irr(B_p(\pi_i(G/C)))\subseteq\Irr(B_p(G))$ is a set of characters of $q'$-degree. Note also that $G/C$ is isomorphic to a subgroup of $\pi_1(G/C)\times\cdots\times\pi_s(G/C)$.  

On the other hand, $\pi_i(G/C)$ is an almost simple group with socle $X_i$. 
It follows from Theorem \ref{as} that 
$\pi_i(G/C)$ has a normal subgroup $T_i$ such that $X_i\leq T_i$, $T_i/X_i$ is a $\{p,q\}'$-group  and $\pi_i(G/C)/T_i$ has a normal abelian Sylow $q$-subgroup. Therefore, $\frac{\pi_1(G/C)\times\cdots\times\pi_s(G/C)}{X_1\times\cdots\times X_s}$ is an extension of a $\{p,q\}'$-group 
by a group with a normal abelian Sylow $q$-subgroup. Since $G/M$ is a subgroup of this group, we conclude that the same holds for $G/M$. Let $T/M\trianglelefteq G/M$ be a $\{p,q\}'$-group such that $G/T$ has a normal abelian Sylow $q$-subgroup. By Step 1 $\bO^{q'}(G/T)=G/T$. We deduce that $G/T$ is an abelian $q$-group, as wanted.

 
 \medskip
 {\it Step 5:} $q$ does not divide $|C|$ and $C$ has a unique $p$-block.
 \medskip

 Since $E$ contains all the components of $G$, and $E\cap C=Z$, it follows that $\bE(C)=1$ so $\bF^*(C)=\bF(C)=\bO_p(C)$. By  \cite[Thm.~9.8]{isa2} we conclude that $\bC_C(\bO_p(C))\subseteq \bO_p(C)$  so $C$ contains a unique $p$-block again by  \cite[Cor.~ V.3.11]{fei}.
 Hence $q$ does not divide the degree of any irreducible character in $C$ by Lemma \ref{step1}. Now,  the It\^o--Michler theorem implies that  $C$ has a normal Sylow $q$-subgroup.
By Step 1, $\bO_q(C)=1$ so
  $q$ does not divide $|C|$.

  \medskip
  {\it Step 6:} Final Step.
  \medskip

By Step 1, $Z$ is a $p$-group. 
 Suppose first that $Z>1$. Then there exists a quasisimple subnormal subgroup $X$ of $G$ with $\bZ(X)$ a nontrivial $p$-group. By Step 3, $X$ is normal in $G$. Let $L\leq\bZ(X)$ be a minimal normal subgroup of $G$.  Let $D=\bC_G(L)$. We claim that $G=D$. By way of contradiction, suppose that $D<G$. By Step 1, every nontrivial factor group of $G/D$ has order divisible by $q$. 
 
 Suppose first that $D$ is cyclic. Therefore, $G/D$ is an abelian $p'$-group that acts faithfully on $L$. By \cite[Prop.~12.1]{mw} it also acts faithfully on $\Irr(L)$. Thus, there exists $\lambda\in\Irr(L)$ in a $G/D$-orbit of size divisible by $q$. It follows from Clifford's correspondence that $q$ divides the degree of any character in $\Irr(G|\lambda)$. But since $\lambda$ belongs to the principal $p$-block of the $p$-group $L$, some of these characters belong to $B_p(G)$. This contradicts the hypothesis. 
 
Hence, we may assume that $L$ is not cyclic. It follows from \cite{atl} that $L=\CCC_p\times\CCC_p$ for $p=2$ or $3$.

Suppose that $p=2$. It follows that  $q=3$ and $G/D\cong\CCC_3$. We reach a contradiction as in the case $L$ cyclic. 

Finally, suppose that $p=3$, so that $G/D$ is isomorphic to a subgroup of $\GL_2(3)$. Hence $q=2$. Arguing as before, we may assume that every character in $\Irr(L)$ is invariant by a Sylow $3$-subgroup of $G/D$. It follows from \cite[Thms.~9.3 and 10.4]{mw} that $G/D\cong\GL_2(3)$. Note that $B_3(G/D)$ has an irreducible character of degree $2$. Since $\Irr(B_0(G/D))\subseteq\Irr(B_0(G))$, this contradicts the hypothesis. This completes the proof of the claim. 

We have thus seen that $L$ is central in $G$. By the inductive hypothesis, $G/L$ has a $p$-nilpotent Hall $\{p,q\}$-subgroup. Let $P\in\Syl_p(G)$ and $Q\in\Syl_q(G)$ be such that $PQ/L$ is a $p$-nilpotent Hall subgroup of $G/L$ (so $PQ$ is a Hall $\{p,q\}$-subgroup of $G$).  Therefore, $QL\trianglelefteq PQ$.   Since $L$ is central in $G$,  $Q$ is characteristic in $QL$ and  we conclude that $Q$ is normal in $PQ$, as desired.

 Suppose now that $Z=1$, so that $M=C\times E$. Therefore, $M/E\cong C$. Since $T/E$ is a $q'$-group  by Steps 4 and 5, we deduce using the Schur-Zassenhaus theorem again that $G/E$ is the semidirect product of $R/E$ acting on $T/E$, where $R/E$ is a Sylow $q$-subgroup of $G/E$. Therefore, $G/M$ is the semidirect product of the $q$-group  $RM/M$ acting on the $\{p,q\}'$-group $T/M$. Since $p$ does not divide $|T/M|$, $\Irr(B_p(T/M))=\Irr(T/M)$. Further, $B_p(G/M)\subseteq B_p(G)$, so $q$ does not divide the degree of any character in $\Irr(B_p(G/M))$. By Lemma \ref{copr}, we conclude that every character in $\Irr(T/M)$ is $G$-invariant. By coprime action (see \cite[Lem~3.2]{isa3}, for instance), this implies that the action of $RM/M$ on $T/M$ is trivial. In particular, $RM\trianglelefteq G$. By Lemma \ref{step1}, $q$ does not divide the degree of any character in $\Irr(B_p(RM))$.

Now, note that $RM/E$ is the semidirect product of the $q$-group $R/E$  acting on the $q'$-group $M/E$. 
 On the other hand, $M/E\cong C$ has a unique $p$-block, by Step 5,  so using Lemma \ref{copr} again, we obtain that all irreducible characters
 of $M/E$ are $R/E$-invariant. As before, this implies that 
$R/E$ acts trivially on $M/E$. We deduce that $R/E\cong RM/M$ acts trivially and coprimely on both $T/M$ and $M/E$. It follows  from \cite[Lem. ~4.29]{isa} that $R/E$ acts trivially on $T/E$. In particular, $R\trianglelefteq G$.
 Since $G/R\cong T/E$ is a $q'$-group and $G=\bO^{q'}(G)$, this implies that $G=R$.

 It follows that $G/E$ is a $q$-group. By Step 4, this implies that  $C=1$. Hence  $G$  is isomorphic to a subgroup  of $\Aut(X_1)\times\cdots\times\Aut(X_s)$ that contains $E=X_1\times\cdots\times X_s$.
 Therefore, if $Q_i/X_i\in\Syl_q(\Aut(X_i)/X_i)$ for every $i$, then $G$ is isomorphic to a subgroup of $Q_1\times\cdots\times Q_s$.

Let $\pi_i: G\longrightarrow Q_i$ be the restriction to $G$ of the projection from  $Q_1\times\cdots\times Q_s$ onto $Q_i$, so that $G/\Ker\pi_i\cong\pi_i(G)$ and $X_i\leq \pi_i(G)\leq Q_i$. Notice that $G$ is isomorphic to a subgroup of $\Gamma=\pi_1(G)\times\cdots\times\pi_s(G)$.

Since $B_p(G/\Ker\pi_i)\subseteq B_p(G)$, $q$ does not divide $\chi(1)$ for every $\chi\in\Irr(B_p(\pi_i(G)))$.

 By Theorem \ref{alt}, this implies that the socle of $\pi_i(G)$ is not an alternating group.
 In the remaining cases, when the socle of $\pi_i(G)$ is a sporadic group or a group of Lie type, Theorems \ref{sporadic} and \ref{lie} imply that $\pi_i(G)$ has a $p$-nilpotent Hall $\{p,q\}$-subgroup. Therefore, $\Gamma$ has a $p$-nilpotent Hall $\{p,q\}$-subgroup. Since $\Gamma/E$ is a $q$-group and $G/E\leq\Gamma/E$, we conclude that $G$ is subnormal in $\Gamma$. By Lemma \ref{sub}, it follows that $G$ has a $p$-nilpotent Hall $\{p,q\}$-subgroup.
 \end{proof}

Next, we conclude the proof of Theorem A, which we restate, with $p$ and $q$ interchanged. As mentioned in the Introduction, it is an easy consequence of Theorem C.

\begin{thm}
Let $G$ be a finite group, let $q$ be a prime and let $Q\in\Syl_q(G)$.
Then $Q\trianglelefteq G$ if and only if for every prime $p\neq q$ that divides $|G|$ and every $\chi\in\Irr(B_p(G))$, $q$ does not divide $\chi(1)$.
\end{thm}

 \begin{proof}
 Suppose first that $Q$ is a normal Sylow $q$-subgroup of $G$. Then for every prime $p\neq q$, $Q\subseteq\bO_{p'}(G)$ so $\Irr(B_p(G))\subseteq\Irr(G/Q)$ is a set of characters of $q'$-degree. Conversely, suppose that for every $p\neq q$, $q$ does not divide $\chi(1)$ for every $\chi\in\Irr(B_p(G))$. Let $Q\in\Syl_q(G)$.  Assume first that $G$ does not have composition factors isomorphic to any of the 4 sporadic groups listed in Theorem \ref{sporadic}. Then Theorem C implies that for any $p\neq q$ a  Sylow $p$-subgroup of $G$ normalizes $Q$. Thus $p$ does not divide $|G:\bN_G(Q)|$ for any prime $p\neq q$ and we conclude that $G=\bN_G(Q)$, as wanted. Finally, assume that $G$ has some composition factor $S$ isomorphic to $J_1, J_4, M_{11}$ or $M_{22}$. Since the hypothesis is inherited by factor groups and by normal subgroups, it is also inherited by composition factors. Therefore, for any $p\neq q$, $q$ does not divide $\chi(1)$ for every $\chi\in\Irr(B_p(S))$. This contradicts Theorem \ref{sporadic}. (For instance, if $S=J_1$ and $q=3$, we can take $p=7$.)
 \end{proof}

Finally, we prove Corollary B.

\begin{cor} \label{cor:blockitomich}
Let $G$ be a finite group and let $p$ be a prime. Then $G$ has a normal and abelian Sylow $p$-subgroup if and only if $p$ does not divide $\chi(1)$ for every $\chi$ that belongs to some principal block for some prime divisor of $|G|$.
\end{cor}

 \begin{proof}
 This follows from Theorem A and Brauer's height zero conjecture for principal blocks \cite{mn2}. 
 \end{proof} 

We remark that the union of the irreducible characters in the principal blocks  of a finite group $G$ is usually a proper subset of $\Irr(G)$ (see  \cite[Thm.~3.7]{bz} and  \cite[Thm.~1]{har}). It was proved in \cite{bz} however that if $G$ is simple then $\AAA_{11}$ and $\AAA_{13}$ are the unique simple groups with this property.

\section{Further directions}

  It makes sense to ask what happens if we just consider characters of height zero (as was done for instance in \cite{nrs}, refining the conjecture in \cite{lwxz}). For instance, if $G=\PSL_3(5)$ or $\SSS_5$, $p=2$ and $q=3$, then $q$ does not divide the degree of any irreducible character of $p$-height zero in $B_p(G)$ but $G$ does not have a $p$-nilpotent Hall $\{p,q\}$-subgroup. It may be worth remarking that there do not seem to be many counterexamples among simple groups, and $\PSL_3(5)$ is the smallest one.

We do not know any counterexamples to the following.

\begin{con}
Let $G$ be a finite group and let $q$ be a prime. Then $G$ has a normal Sylow $q$-subgroup if and only if
for any $p\neq q$, $q$ does not divide the degree of any $p$-height zero irreducible character in $B_p(G)$.
\end{con}

We can prove this conjecture, and the $p$-height zero version of Theorem C, for solvable groups.

\begin{pro}
Let $p$ and $q$ be two different primes.
Let $G$ be a $\{p,q\}$-solvable group. If $q$ does not divide the degree of any $p$-height zero irreducible character in $B_p(G)$, then  a Sylow $p$-subgroup of $G$ normalizes a Sylow $q$-subgroup of $G$. In particular, if $G$ is solvable then  $G$ has a normal Sylow $q$-subgroup if and only if for every prime $p\neq q$ that divides $|G|$, $q$ does not divide $\chi(1)$ for every $\chi\in\Irr(B_p(G))$ of $p$-height zero.
\end{pro}

\begin{proof}
Since $G$ is $p$-solvable, $\Irr(B_p(G))=\Irr(G/\bO_{p'}(G))$. Therefore, the hypothesis implies that $q$ does not divide the degree of any irreducible character of $p'$-degree of $G/\bO_{p'}(G)$. Now, the first part follows from   Theorem A of \cite{nw2}.  For the second part we can argue as in the proof of Theorem A.
\end{proof}

Recently, along with Malle and Rizo, we have proposed a strengthening of Brauer's height zero conjecture with Galois automorphisms \cite{mmrs}. We think that it should be possible to strengthen Theorem A in a similar way. Given a prime $p$,  let $\cJ_p$ be the subgroup of $\Gal(\QQ^{\rm{ab}}/\QQ)$  consisting of the automorphisms of  order $p$ that fix all $p$-power order roots of unity. Given a group $G$, let $\Irr_{\cJ_p}(B_p(G))$ be the set of $\cJ_p$-invariant irreducible characters of the principal $p$-block of $G$.

\begin{con}
Let $G$ be a finite group and let $q$ be a prime. Then $G$ has a normal Sylow $q$-subgroup if and only if
for any $p\neq q$ that divides $|G|$, $q$ does not divide $\chi(1)$ for every $\chi\in\Irr_{\cJ_p}(B_p(G))$.
\end{con}

Again, we can prove this conjecture for solvable groups.

\begin{pro}
Let $G$ be a $p$-solvable group. If $q$ does not divide the degree of any character in $\Irr_{\cJ_p}(B_p(G))$ then  a Sylow $p$-subgroup of $G$ normalizes a Sylow 
$q$-subgroup of $G$. In particular, if $G$ is solvable, then $G$ has a normal Sylow $q$-subgroup if and only if for every prime $p\neq q$ that divides $|G|$, $q$ does not divide $\chi(1)$ for every $\chi\in\Irr_{\cJ_p}(B_p(G))$.
\end{pro}

\begin{proof}
Since $G$ is $p$-solvable, $\Irr(B_p(G))=\Irr(G/\bO_{p'}(G))$. Hence, $q$ does not divide the degree of any character in $\Irr_{\cJ_p}(G/\bO_{p'}(G))$. By \cite[Thm 3]{mmrs}, this implies that $G/\bO_{p'}(G)$ has a normal Sylow $q$-subgroup. Thus $q$ does not divide $G/\bO_{p'}(G)$. Let $P$ be a Sylow $p$-subgroup of $G$. Since $P$ acts coprimely on $\bO_{p'}(G)$, using \cite[Thm.~3.23]{isa2} we conclude that there exists a $P$-invariant Sylow $q$-subgroup of $G$. The first part of the result follows. For the second part we can argue as in the proof of Theorem A.
\end{proof}

It makes sense to ask if there is a version of Brauer's height zero conjecture for an arbitrary number of primes. This has an easier solution,  as a consequence of Brauer's height zero conjecture for two primes and \cite{mor}.

\begin{thm}
Let $G$ be a finite group and let $\pi=\{p_1,\dots,p_n\}$ be a set of primes. Then $G$ has a nilpotent Hall $\pi$-subgroup if and only if for every $i$ and every $j\neq i$, $p_j$ does not divide the degree of any irreducible character in $B_{p_i}(G)$.
\end{thm}

\begin{proof}
Let $\tau=\{p_i,p_j\}$ be a subset of cardinality $2$ of $\pi$.
The ``only if" part follows immediately from the fact that if $G$ has a nilpotent Hall $\pi$-subgroup, then $G$ has a nilpotent Hall $\tau$-subgroup and the ``only if" part of Brauer's height zero conjecture for two primes \cite{mn}. Conversely, if for every $i$ and every $j\neq i$, $p_j$ does not divide the degree of any irreducible character in $B_{p_i}(G)$, then Theorem \ref{thm:2pBHZ}  implies that $G$ has nilpotent Hall $\tau$-subgroups for every $\tau\subseteq\pi$ with $|\tau|=2$. Now, we deduce that $G$ has a nilpotent Hall $\pi$-subgroup using  \cite[Lem.~3.4]{mor} (which relies on the classification) and the comment that follows it.
\end{proof}

\end{document}